\documentclass[reqno]{amsart}

\usepackage{amssymb}
\usepackage{hyperref}
\usepackage{ifthen}
\usepackage{xargs}
\usepackage{mathtools}

\hypersetup{colorlinks=true,linkcolor=blue}

\newcommand{\axitem}[1]{\phantomsection \label{ax:#1}}
\newcommand{\axref}[1]{(\hyperref[ax:#1]{#1})}

\newcommand{\newref}[4][]{
\ifthenelse{\equal{#1}{}}{\newtheorem{h#2}[hthm]{#4}}{\newtheorem{h#2}{#4}[#1]}
\expandafter\newcommand\csname r#2\endcsname[1]{#3~\ref{#2:##1}}
\expandafter\newcommand\csname R#2\endcsname[1]{#4~\ref{#2:##1}}
\expandafter\newcommand\csname n#2\endcsname[1]{\ref{#2:##1}}
\newenvironmentx{#2}[2][1=,2=]{
\ifthenelse{\equal{##2}{}}{\begin{h#2}}{\begin{h#2}[##2]}
\ifthenelse{\equal{##1}{}}{}{\label{#2:##1}}
}{\end{h#2}}
}

\newref[section]{thm}{Theorem}{Theorem}
\newref{lem}{Lemma}{Lemma}
\newref{prop}{Proposition}{Proposition}
\newref{cor}{Corollary}{Corollary}
\newref{cond}{Condition}{Condition}

\theoremstyle{definition}
\newref{defn}{Definition}{Definition}
\newref{example}{Example}{Example}

\theoremstyle{remark}
\newref{remark}{Remark}{Remark}

\numberwithin{figure}{section}

\newcommand{\overlap}[2]{#1 \between #2}
\newcommand{\rb}{\prec}
\newcommand{\cat}[1]{\mathbf{#1}}

\DeclarePairedDelimiter\abs{\lvert}{\rvert}

\makeatletter
\let\oldabs\abs
\def\abs{\@ifstar{\oldabs}{\oldabs*}}
\makeatother

\begin{document}

\title{A Constructive Approach to Complete Spaces}

\author{Valery Isaev}

\begin{abstract}
In this paper, we present a constructive generalization of metric and uniform spaces by introducing a new class of spaces, called cover spaces.
These spaces form a topological concrete category with a full reflective subcategory of complete spaces.
This subcategory is closely related to a particular subcategory of locales, offering an alternative approach to localic completion.
Additionally, we demonstrate how this framework provides simple constructive definitions of compact spaces, uniform convergence, and limits of nets.
\end{abstract}

\maketitle

\section{Introduction}

This paper explores completeness, continuous mappings, and other topological concepts within the framework of constructive mathematics.
The traditional topological notion of continuity between real numbers often proves unsuitable in a constructive context.
For instance, Bishop defined continuous maps as locally uniformly continuous maps \cite[Chapter 2, Definition 9]{bishop},
a perspective that does not generalize well to spaces that are not locally compact.

Another prevalent constructive approach to general topology involves locale theory.
Palmgren demonstrated that locally uniform maps between real numbers correspond to localic endomorphisms on the locale of real numbers \cite{palmgren-cont}.
Additionally, locally compact metric spaces can be fully and faithfully embedded into the category of locales \cite{palmgren-metric-locales}.
In this paper, we propose an alternative framework that generalizes these results to spaces that are not locally compact.

We introduce a topological concrete category of \emph{cover spaces}.
Every metric space, and more generally, every uniform space, can naturally be regarded as a cover space.
We establish the existence of a full reflective subcategory of complete cover spaces (\rcor{complete-reflective}),
and we show that a particular full subcategory of locales is equivalent to a full subcategory of complete cover spaces.
This equivalence encompasses a broad class of complete metric spaces (\rthm{locale-equiv}),
offering a constructive alternative to localic completion \cite{localic-completion}.

The definition of cover spaces allows us to generalize the Cauchy condition for sequences and, more broadly, for nets (\rprop{conv-char}).
This generalization provides a motivation for the concept of cover spaces.
Classically, if $x : \mathbb{N} \to \mathbb{R}$ is a Cauchy sequence, then every open cover of $\mathbb{R}$ contains a set that eventually includes all $x_n$.
However, this property does not hold for sequences of rational numbers.
For example, if $x : \mathbb{N} \to \mathbb{Q}$ converges to $\sqrt{2}$ from both sides, then the cover $\{ (-\infty,\sqrt{2}), (\sqrt{2},\infty) \}$ does not satisfy the condition.
Informally, the problem lies in the fact that this cover does not cover all the points in the completion of $\mathbb{Q}$.
We can resolve this by augmenting the cover with additional sets.
For example, the cover ${ (-\infty, \sqrt{2}), (\sqrt{2}, \infty), (1,2) }$ satisfies the condition for all Cauchy sequences.

In general, we can choose a class of good covers of $\mathbb{Q}$, which we call \emph{Cauchy covers},
such that a sequence is Cauchy if and only if every Cauchy cover contains a set that the sequence eventually belongs to.
A \emph{cover space} is a set equipped with a choice of Cauchy covers satisfying certain closure conditions.
A \emph{complete cover space} is one in which every Cauchy sequence (or filter) converges to a point.
Constructively, there may be fewer Cauchy covers in a complete cover space than neighborhood covers, making the concept of a Cauchy cover significant even for complete spaces.

Cauchy covers share similarities with the cover relation in locale theory and can be utilized to constructively reformulate various classical definitions and theorems.
For instance, a complete cover space can be defined as \emph{compact} if every Cauchy cover admits a finite subcover.
Under this notion of compactness, we will prove constructive versions of the Heine-Borel theorem and the Tychonoff theorem.

Cover spaces provide a convenient framework for defining convergence for arbitrary nets.
In Section~\ref{sec:limits}, we demonstrate how to put a cover space structure on an arbitrary directed set $I$ in such a way that a map $I \to X$ converges if and only if it is a morphism of cover spaces.
We also show how Cauchy covers can be used to derive simple and elegant conditions for the convergence of a sequence of functions, thereby generalizing uniform convergence.

The paper is organized as follows:
\begin{itemize}
\item Section~\ref{sec:top} defines cover spaces and explores their topological properties.
\item Section~\ref{sec:cat} presents the category of cover spaces and its basic properties.
\item Section~\ref{sec:filters} introduces Cauchy filters, which are essential for defining complete cover spaces.
\item Section~\ref{sec:completion} defines the completion of cover spaces and proves its universal property.
\item Section~\ref{sec:strongly-regular}, establishes a strong regularity and completeness properties, which are instrumental in relating cover spaces to sober topological spaces and locales.
\item Section~\ref{sec:locales} constructs an equivalence between certain full subcategories of locales and strongly complete cover spaces.
\item Section~\ref{sec:reals} examines the cover space of real numbers in detail.
\item Section~\ref{sec:compact} defines compact cover spaces and proves that morphisms of locally compact cover spaces are precisely locally uniform maps.
\item Section~\ref{sec:limits} discusses limits in cover spaces and provides several criteria for their existence.
\end{itemize}

All results presented in this paper are constructive, meaning they hold within the internal language of an elementary topos.
Furthermore, these results have been formalized in the Arend proof assistant.

\section{Topological properties}
\label{sec:top}

In this section, we define cover spaces and study their topological properties.
We begin with a few basic definitions:
\begin{itemize}
\item Two subsets $U,V$ of a set $X$ are said to \emph{intersect}, written $\overlap{U}{V}$, if there exists an element of $X$ that belongs to both subsets.
\item A set $C$ of subsets of a set $X$ is called a \emph{cover} if $\bigcup C = X$. If $C$ is a cover of $X$, we also say that $C$ \emph{covers} $X$ or that $X$ \emph{is covered by} $C$.
\item A cover $C$ \emph{refines} a cover $D$ if every element of $C$ is a subset of some element of $D$.
\item If $U$ and $V$ are subsets of a topological space $X$, we say that $V$ is \emph{rather below} $U$, written $V \rb_X U$ or, for short, $V \rb U$, if $X$ is covered by open sets $W$ satisfying the implication $\overlap{W}{V} \implies W \subseteq U$.
\item A topological space is \emph{regular} if every open set $U$ is covered by open sets that are rather below $U$.
\end{itemize}

Now, we are ready to define cover spaces:
\begin{defn}
A \emph{cover space} is a set $X$ equipped with a collection $\mathcal{C}_X$ of covers of $X$ that satisfy the following conditions:
\begin{itemize}
\item[(CT)] \axitem{CT} The trivial cover $\{ X \}$ belongs to $\mathcal{C}_X$.
\item[(CE)] \axitem{CE} If $C \in \mathcal{C}_X$ and $C$ refines $D$, then $D \in \mathcal{C}_X$
\item[(CG)] \axitem{CG} If $C \in \mathcal{C}_X$ and $\{ D_U \}_{U \in C}$ is a collection of covers such that $D_U \in \mathcal{C}_X$ for all $U \in C$,
            then $\{ U \cap V \mid U \in C, V \in D_U \} \in \mathcal{C}_X$.
\item[(CR)] \axitem{CR} If $C \in \mathcal{C}_X$, then $\{ V \subseteq X \mid \exists U \in C, V \rb_X U \} \in \mathcal{C}_X$,
            where $V \rb_X U$ (or, for short, $V \rb U$) means that $\{ W \subseteq X \mid \overlap{W}{V} \implies W \subseteq U \} \in \mathcal{C}_X$.
\end{itemize}
Elements of $\mathcal{C}_X$ are called \emph{Cauchy covers} of $X$.
\end{defn}

\begin{remark}
Under the assumption of the law of excluded middle (LEM), the relation $V \rb U$ can equivalently be expressed as $\{ X \backslash V, U \} \in \mathcal{C}_X$.
Indeed, $\{ X \backslash V, U \}$ is always a subset of $\{ W \mid \overlap{W}{V} \implies W \subseteq U \}$.
Assuming LEM, the latter cover also refines the former.
We will discuss the condition $\{ X \backslash V, U \} \in \mathcal{C}_X$ in greater detail in Section~\ref{sec:completion}.
\end{remark}

The relation $\rb$ satisfies the usual properties of the ``rather below'' relation:

\begin{prop}[rb-props]
The following hold for subsets $U,V,U',V'$ of a cover space $X$:
\begin{enumerate}
\item \label{rb:sub} If $V \rb U$, then $V \subseteq U$.
\item If $V' \subseteq V \rb U \subseteq U'$, then $V' \rb U'$.
\item \label{rb:meet} If $V \rb U$ and $V' \rb U'$, then $V \cap V' \rb U \cap U'$.
\item $V \rb X$.
\item $\varnothing \rb U$.
\end{enumerate}
\end{prop}
\begin{proof}
Item~\eqref{rb:sub} follows from the fact that $\{ W \mid \overlap{W}{V} \implies W \subseteq U \}$ is a cover.
Item~\eqref{rb:meet} follows from \axref{CG}.
The remaining statements are straightforward.
\end{proof}

\begin{example}
For a set $X$, the \emph{indiscrete} cover space structure consists of all covers that include the whole set: $\mathcal{C}_X = \{ C \mid X \in C \}$.
This is the minimal set of covers that makes $X$ into a cover space.
\end{example}

\begin{example}
For a set $X$, the \emph{discrete} cover space structure consists of all covers.
This is the maximal set of covers that makes $X$ into a cover space.
\end{example}

\begin{example}
More generally, if $X$ is a regular topological space, a cover is defined to be Cauchy if it contains a neighborhood of every point.
Under this definition, the rather below relations for $X$ as a topological space coincides with that for $X$ as a cover space.
Thus, the regularity condition on $X$ ensures \axref{CR}, while the other axioms of cover spaces are trivially satisfied.
We denote this cover space by $T(X)$.
\end{example}

We can define various basic topological concepts for cover spaces:

\begin{defn}[cover-topology]
Let $X$ be a cover space.
\begin{itemize}
\item A \emph{neighborhood} of a point $x \in X$ is a subset $U$ of $X$ such that $\{ x \} \rb U$.
\item A subset $U \subseteq X$ is \emph{open} if $U$ is a neighborhood of each point $x \in U$.
\item A \emph{limit point} of a subset $U \subseteq X$ is a point $x \in X$ such that every neighborhood of $x$ intersects $U$.
\item A subset of $X$ is \emph{closed} if it contains all of its limit points.
\item The \emph{closure} $\overline{U}$ of a subset $U \subseteq X$ is the set of its limit points.
\item A subset $D \subseteq X$ is \emph{dense} if every point of $X$ is a limit point of $D$.
\end{itemize}
\end{defn}

From \rprop{rb-props}, it follows that the collection of open subsets defined above forms a topology on every cover space $X$.
We denote this topological space by $S(X)$.

\begin{example}
Let $X$ be an indiscrete cover space.
A subset $U \subseteq X$ is a neighborhood of a point $x \in X$ if and only if $U = X$.
Thus, a subset $U \subseteq X$ is open if and only if $U = X$ whenever it is inhabited.
It follows that $S(X)$ is an indiscrete topological space.
\end{example}

\begin{example}
Let $X$ be a discrete cover space.
A subset $U \subseteq X$ is a neighborhood of a point $x \in X$ if and only if $x \in U$.
Hence, every subset of $X$ is open, making $S(X)$ a discrete topological space.
\end{example}

\begin{example}
Let $X$ be a regular topological space.
The regularity of $X$ implies that a subset $U \subseteq X$ is a neighborhood of a point $x \in X$ in the topological sense if and only if $\{ x \} \rb_{T(X)} U$.
Thus, the topologoy $S(T(X))$ induced by the cover space $T(X)$ coincides with the original topology of $X$.
\end{example}

The following lemma is useful for establishing the relationship between a cover space $X$ and the corresponding topological space $S(X)$:

\begin{lem}[int-char]
Let $U$ be a subset of a cover space $X$.
Then the interior of $U$ is equal to $\{ x \mid \{ x \} \rb_X U \}$.
In particular, if $V \rb_X U$, then $V \subseteq \mathrm{int}(U)$.
\end{lem}
\begin{proof}
If $x \in \mathrm{int}(U)$, then $\{ x \} \rb_X \mathrm{int}(U) \subseteq U$, so one direction is clear.
Conversely, let $U' = \{ x \mid \{ x \} \rb_X U \}$.
We will show that $U'$ is open and $U' \subseteq \mathrm{int}(U)$.

If $x \in U'$, then by definition $\{ x \} \rb_X U$.
Consider the collection $\{ W' \mid \exists W, W' \rb_X W, x \in W \implies W \subseteq U \}$, which forms a Cauchy cover refining $\{ W' \mid x \in W' \implies W' \subseteq U' \}$.
This implies that $\{ x \} \rb_X U'$, showing that $U'$ is open.

Since $\mathrm{int}(U)$ is the largest open subset of $U$, we have $U' \subseteq \mathrm{int}(U)$. Hence, $\mathrm{int}(U) = \{ x \mid \{ x \} \rb_X U \}$.
\end{proof}

The following proposition shows that every Cauchy cover is refined by an open Cauchy cover.
This means we can restrict our attention to open covers if preferred, as some may find them more convenient,
but since it makes no substantive difference, we will continue to work with arbitrary covers.

\begin{prop}[cover-int]
If $C$ is a Cauchy cover, then $\{ \mathrm{int}(U) \mid U \in C \}$ is also a Cauchy cover.
\end{prop}
\begin{proof}
If $C$ is a Cauchy cover, then $\{ V \mid \exists U \in C, V \rb U \}$ is also Cauchy.
By \rlem{int-char}, this cover refines $\{ \mathrm{int}(U) \mid U \in C \}$.
Thus, $\{ \mathrm{int}(U) \mid U \in C \}$ is a Cauchy cover.
\end{proof}

Now, we will show how various topological properties of a cover space $X$ relate to corresponding properties of $S(X)$:

\begin{prop}[top-neighborhood]
If $X$ is a cover space, then the notions of a neighborhood, closed subsets, closure, and dense subsets for $X$ coincide with the corresponding properties of $S(X)$.
\end{prop}
\begin{proof}
It is enough to prove the result for neighborhoods since the other properties are defined in terms of neighborhoods.
In $S(X)$, a set $U$ is a neighborhood of a point $x$ if there exists an open set $V$ such that $x \in V \subseteq U$.
If $V$ is such a set, then $\{ x \} \rb_X V \subseteq U$.
Conversely, we can take $V = \mathrm{int}(U)$ since $\{ x \} \rb_X U$ implies that $x \in \mathrm{int}(U) \subseteq U$ by \rlem{int-char}.
\end{proof}

Next, we show that $S(X)$ is always a regular topological space.
To do this, we first prove a useful lemma:

\begin{lem}[rb-point]
Let $X$ be a cover space.
If $U$ is a neighborhood of a point $x$, then there exists a neighborhood $V$ of $x$ such that $V \rb U$.
\end{lem}
\begin{proof}
If $\{ x \} \rb U$, then $\{ V' \mid \exists V, \exists W, V' \rb V \rb W, x \in W \implies W \subseteq U \}$ covers $X$.
Thus, there exist subsets $V'$, $V$, and $W$ such that $x \in V' \rb V$ and $V \rb W \subseteq U$.
So, $V$ is the required neighborhood.
\end{proof}

\begin{prop}[top-regular]
For every cover space $X$, the topological space $S(X)$ is regular.
\end{prop}
\begin{proof}
First, note that if $V \rb_X U$, then $\{ \mathrm{int}(W) \mid x \in W \implies W \subseteq U \}$ covers $X$ by \rprop{cover-int},
so, for every $x \in X$, we have a set $W$ such that $x \in \mathrm{int}(W)$ and $\overlap{\mathrm{int}(W)}{V} \implies \overlap{W}{V} \implies \mathrm{int}(W) \subseteq W \subseteq U$.
It follows that $V \rb_{S(X)} U$.

To see that $S(X)$ is regular, consider an open set $U$.
By \rlem{rb-point}, there exists a neighborhood $V$ of $x$ such that $V \rb_X U$.
The previous observation implies that $V \rb_{S(X)} U$.
By \rprop{top-neighborhood}, there is an open neighborhood $V'$ of $x$ such that $V' \rb_{S(X)} U$.
Thus, $S(X)$ is regular.
\end{proof}

The following proposition will be useful later.
It shows that $V \rb U$ implies the usual definition of the rather below relation in terms of limit points:

\begin{prop}[rb-closure]
Let $U,V$ be subsets of a cover space $X$ such that $V \rb U$.
Then $U$ is a neighborhood of every limit point of $V$.
In particular, the closure of $V$ is a subset of $U$, and $\overline{V} \rb \mathrm{int}(U)$.
\end{prop}
\begin{proof}
We prove the last claim.
By \rprop{cover-int}, the cover $\{ \mathrm{int}(W') \mid \exists W, W' \rb W, \overlap{W}{V} \implies W \subseteq U \}$ is Cauchy.
This cover refines $\{ W \mid \overlap{W}{\overline{V}} \implies W \subseteq \mathrm{int}(U) \}$.
Thus, $\overline{V} \rb \mathrm{int}(U)$.
\end{proof}

Many examples of cover spaces arise naturally from uniform spaces.
Recall that a uniform space is a set $X$ equipped with a collection of covers $\mathcal{U}_X$ that satisfies the following axioms:

\begin{itemize}
\item[(UT)] The trivial cover $\{ X \}$ belongs to $\mathcal{U}_X$.
\item[(UE)] \axitem{UE} If $C \in \mathcal{U}_X$ and $C$ refines $D$, then $D \in \mathcal{U}_X$.
\item[(UI)] \axitem{UI} If $C,D \in \mathcal{U}_X$, then $\{ U \cap V \mid U \in C, V \in D \} \in \mathcal{U}_X$.
\item[(UU)] \axitem{UU} For every $C \in \mathcal{U}_X$, there exists $D \in \mathcal{U}_X$ such that, for every $V \in D$,
there exists $U \in C$ such that $D \subseteq \{ W \mid \overlap{W}{V} \implies W \subseteq U \}$.
\end{itemize}
The elements of $\mathcal{U}_X$ are called \emph{uniform covers} of $X$.

\begin{example}
Every metric space induces a uniform structure.
Specifically, a cover of a metric space $X$ is uniform if it refines the cover $\{ B_\varepsilon(x) \mid x \in X \}$ for some $\varepsilon > 0$.
\end{example}

We now demonstrate how every uniform space gives rise to a cover space.
To achieve this, we introduce the intermediate notion of a \emph{precover space}:

\begin{defn}
A \emph{precover space} is a set $X$ together with a collection $\mathcal{C}_X$ of covers of $X$ that satisfies the conditions \axref{CT}, \axref{CE}, and \axref{CG}.
\end{defn}

In this framework, a \emph{cover space} is precisely a precover space that additionally satisfies the regularity condition \axref{CR}.

We also define the notion of a \emph{subbase} for a precover structure:

\begin{defn}
A \emph{subbase} for a precover structure on a set $X$ is any collection of covers of $X$.
\end{defn}

Given a subbase $\mathcal{B}_X$, we can take its closure under \axref{CT}, \axref{CE}, and \axref{CG}.
We denote this closure by $\overline{\mathcal{B}_X}$.
The following lemma applied to the filter of sets containing a given point implies that all sets in this closure are covers:

\begin{lem}[closure-filter]
Let $\mathcal{B}_X$ be a subbase and $F$ a filter on $X$.
If $F$ intersects every set in $\mathcal{B}_X$, then it also intersects every set in $\overline{\mathcal{B}_X}$.
\end{lem}
\begin{proof}
The proof follows by induction on the steps used to generate $\overline{\mathcal{B}_X}$.
\end{proof}

Thus, $\overline{\mathcal{B}_X}$ is the minimal precover space structure that contains $\mathcal{B}_X$.
However, the pair $(X, \overline{\mathcal{B}_X})$ is not necessarily a cover space.
The following lemma provides a useful tool for extending regularity conditions from a subbase to the entire space:

\begin{lem}[subbase-regular]
Let $\mathcal{B}_X$ be a subbase, and let $\rb$ be a binary relation on subsets of $X$ satisfying the following conditions:
\begin{itemize}
\item $X \rb X$.
\item If $W \rb V \subseteq U$, then $W \rb U$ for all subsets $W,V,U$.
\item If $V \rb U$ and $V' \rb U'$, then $V \cap V' \rb U \cap U'$ for all subsets $V,V',U,U'$.
\end{itemize}
Suppose that, for every cover $C \in \mathcal{B}_X$, the collection $\{ V \mid \exists U \in C, V \rb U \}$ belongs to $\overline{\mathcal{B}_X}$.
Then the same property holds for every cover $C \in \overline{\mathcal{B}_X}$.
\end{lem}
\begin{proof}
We prove this by induction on the generation of $\overline{\mathcal{B}_X}$.
Specifically, we show that for every $C \in \overline{\mathcal{B}_X}$, the cover $\{ W' \mid \exists W \in C, W' \rb W \}$ also belongs to $\overline{\mathcal{B}_X}$.
The only non-trivial case arises from the axiom \axref{CG}.

Assume $C = \{ U \cap V \mid U \in D, V \in E_U \} \in \overline{\mathcal{B}_X}$, where $D \in \overline{\mathcal{B}_X}$ and $\{ E_U \in \overline{\mathcal{B}_X} \}_{U \in D}$.
By the induction hypothesis, we know:
\begin{align*}
D' & = \{ U' \mid \exists U \in D, U' \rb U \} \in \overline{\mathcal{B}_X} \\
E'_{U'} & = \{ V' \mid \exists U \in D, V \in E_U, V' \rb V, U' \rb U \} \in \overline{\mathcal{B}_X} \text{ for every } U' \in D'
\end{align*}
Since $\{ W' \mid \exists W \in C, W' \rb W \}$ is refined by $\{ U' \cap V' \mid U' \in D', V' \in E'_{U'} \}$, the proof is complete.
\end{proof}

We will see later that a stronger condition on a subbase is often useful:

\begin{defn}
A \emph{base} for a precover space structure is a collection $\mathcal{B}$ of covers of $X$ satisfying the following axioms:
\begin{itemize}
\item[(BT)] \axitem{BT} The trivial cover $\{ X \}$ belongs to $\mathcal{B}$.
\item[(BE)] \axitem{BE} If $C \in \mathcal{B}$ and $C$ refines $D$, then $D \in \mathcal{B}$.
\item[(BI)] \axitem{BI} If $C,D \in \mathcal{B}$, then $\{ U \cap V \mid U \in C, V \in D \} \in \mathcal{B}$.
\end{itemize}
Let $U,V$ be subsets of $X$.
We say that $V$ is \emph{$\mathcal{B}$-rather below} $U$, written $V \rb_\mathcal{B} U$, if $\{ W \subseteq X \mid \overlap{W}{V} \implies W \subseteq U \} \in \mathcal{B}$.
A base $\mathcal{B}$ is called \emph{regular} if it satisfies the following condition:
\begin{itemize}
\item[(BR)] \axitem{BR} For every $C \in \mathcal{B}$, we have $\{ V \mid \exists U \in C, V \rb_\mathcal{B} U \} \in \overline{\mathcal{B}}$.
\end{itemize}
\end{defn}

\begin{remark}
The conditions of \rlem{subbase-regular} apply to the relation $\rb_\mathcal{B}$.
It follows that the precover space generated by a base $\mathcal{B}$ is not only a cover space but also satisfies a stronger regularity condition:
if $C$ is a Cauchy cover, then $\{ V \mid \exists U \in C, V \rb_\mathcal{B} U \}$ is also a Cauchy cover.
\end{remark}

\begin{example}
The collection of all Cauchy covers of a cover space forms a regular base.
\end{example}

\begin{example}
If $X$ is a uniform space, then the collection of uniform covers is a regular base.
The regularity axiom \axref{BR} follows from \axref{UU}.
\end{example}

The previous example shows that every uniform space $(X, \mathcal{U}_X)$ determines a cover space $(X, \overline{\mathcal{U}_X})$.
In particular, every metric space defines a cover space.
These are fundamental examples of cover spaces.

\section{Categorical properties}
\label{sec:cat}

In this section, we define the category of cover spaces and establish its various properties.
For convenience, we also define the category of precover spaces, a broader framework that will aid in understanding the structure of cover spaces.

We begin with the definition of morphisms between precover spaces:

\begin{defn}
A \emph{cover map} between precover spaces $X$ and $Y$ is a function $f : X \to Y$ such that, for every Cauchy cover $D$ on $Y$, the collection $\{ f^{-1}(V) \mid V \in D \}$ forms a Cauchy cover on $X$.
\end{defn}

The identity function is trivially a cover map, and cover maps are closed under composition.
This makes the collection of precover spaces and cover maps into a category, denoted by $\cat{Precov}$.
The subcategory of $\cat{Precov}$ consisting of cover spaces is called the category of cover spaces, denoted by $\cat{Cov}$.

We now show that several constructions previously defined are functorial.
To facilitate this, we first establish a simple lemma:

\begin{lem}[cover-map-rb]
Let $f : X \to Y$ be a cover map, and let $U$ and $V$ be subsets of $Y$ such that $V \rb_Y U$.
Then $f^{-1}(V) \rb_X f^{-1}(U)$.
\end{lem}
\begin{proof}
The result follows from the fact that the collection $\{ W \mid \overlap{W}{f^{-1}(V)} \implies W \subseteq f^{-1}(U) \}$ is refined by $\{ f^{-1}(W) \mid \overlap{W}{V} \implies W \subseteq U \}$.
\end{proof}

Using this lemma, we show that the constructions $S : \cat{Cov} \to \cat{Top}_\mathrm{reg}$ and $T : \cat{Top}_\mathrm{reg} \to \cat{Cov}$ are functors.
Here, $\cat{Top}_\mathrm{reg}$ denotes the category of regular topological spaces.

\paragraph{Functoriality of $S$.}
Let $f : X \to Y$ be a cover map between cover spaces.
For any open subset $U$ of $Y$, we claim that $f^{-1}(U)$ is open in $X$.
Indeed, given $x \in f^{-1}(U)$, note that $\{x\} \subseteq f^{-1}(\{f(x)\}) \rb f^{-1}(U)$ by \rlem{cover-map-rb},
which implies that $f^{-1}(U)$ satisfies the openness condition in $S(X)$.
Hence, $S(f)$ is a continuous map, demonstrating that $S$ is functorial.

\paragraph{Functoriality of $T$.}
For any regular topological space $X$, the associated functor $T$ assigns to $X$ a cover space structure.
It is clear that $T$ preserves morphisms, and therefore, it is also a functor.

\paragraph{Faithfulness and fullness.}
Both $S$ and $T$ are faithful functors as they preserve the underlying functions of morphisms.
Furthermore, $T$ is full.
Since $S(T(X)) = X$ for every regular topological space $X$, any cover map $f : T(X) \to T(Y)$ induces a continuous map $g = S(f) : X \to Y$ such that $T(g) = f$.

\paragraph{Adjunction.}
The functor $T$ is a left adjoint to $S$, with the unit and counit of the adjunction being identity maps.
The counit $T(S(X)) \to X$ is a cover map since any Cauchy cover contains a neighborhood of every point.
Consequently, $\cat{Top}_\mathrm{reg}$ is equivalent to a full coreflective subcategory of $\cat{Cov}$.

\paragraph{Topological category structure of $\cat{Precov}$.}
Finally, we show that $\cat{Precov}$ is a topological category in the sense of \cite[Definition~21.7]{joy-cats}.
This requires demonstrating that the set of precover space structures on a given set $X$ forms a complete lattice under inclusion.
Given a collection $\{\mathcal{C}_i\}_{i \in I}$ of precover space structures on $X$,
the join of these structures is $\overline{\bigcup_{i \in I} \mathcal{C}_i}$.

\paragraph{Transferred precover structures.}
Let $f : X \to Y$ be a function between sets, and let $Y$ have a precover space structure.
The smallest precover space structure on $X$ for which $f$ is a cover map is constructed as follows:
A cover $C$ of $X$ is Cauchy if and only if $\{ V \mid \exists U \in C, f^{-1}(V) \subseteq U \}$ is a Cauchy cover of $Y$.
These constructions collectively establish that $\cat{Precov}$ is a topological category.

Cover spaces are closed under joins in the lattice of precover space structures on a set $X$.
This property follows directly from \rlem{subbase-regular}.
Additionally, if $Y$ is a cover space and $f : X \to Y$ is a function, then the transferred precover structure on $X$ is, in fact, a cover structure.
Consequently, the category of cover spaces is reflective in $\cat{Precov}$:

\begin{prop}[regular-reflective]
The category of cover spaces is reflective in $\cat{Precov}$.
The reflection $X \to R(X)$ is given by the identity function on $X$.
\end{prop}
\begin{proof}
For a precover space $X$, define its reflection $R(X)$ as the join of all cover space structures on $X$ that are contained within $\mathcal{C}_X$.
The identity function $X \to R(X)$ is a cover map and satisfies the universal property of the reflection.
Specifically, if $Y$ is a cover space and $f : X \to Y$ is a cover map, then the transferred cover structure on $X$ is a subset of $\mathcal{C}_{R(X)}$.
Hence, $f$ is a cover map from $R(X)$ to $Y$.
\end{proof}

\begin{cor}
The category of cover spaces is a topological category.
\end{cor}

We conclude this section with a discussion of embeddings.
A cover map $f : X \to Y$ is called an \emph{embedding} if, for every $C \in \mathcal{C}_X$,
the set $\{ V \mid \exists U \in C, f^{-1}(V) \subseteq U \}$ is a Cauchy cover of $Y$.
Note that a cover map $f : X \to Y$ is an embedding if and only if the precover space structure on $X$ is the transferred precover space structure.

In general, an embedding does not need to be injective.
However, we will show in \rprop{embedding-injective} that injectivity holds if the domain satisfies a separation condition.
For injective embeddings, we have the following characterization:

\begin{prop}
A cover map is a regular monomorphism in $\cat{Precov}$ if and only if it is an injective embedding.
Similarly, a cover map between cover spaces is a regular monomorphism in $\cat{Cov}$ if and only if it is an injective embedding.
\end{prop}
\begin{proof}
Suppose $f : X \to Y$ is a regular monomorphism in $\cat{Precov}$.
Then $f$ is injective, and it is straightforward to verify that the precover space structure on $X$ is the transferred structure.

Conversely, let $f : X \to Y$ be an injective embedding.
Then $f : X \to Y$ is an equalizer of some functions $g,h : Y \to Z$ in $\cat{Set}$.
Assign the indiscrete structure to $Z$, making $g$ and $h$ cover maps, and it is easy to see that $f$ is their equalizer.
The same reasoning applies in $\cat{Cov}$ for cover spaces.
\end{proof}

We will later need the fact that embeddings are closed under products.
To establish this, we first prove the following lemma:

\begin{lem}
Let $\mathcal{B}_X$ be a subbase for a precover space structure on $X$, and let $Y$ be a precover space.
Let $f : X \to Y$ be a cover map such that, for every $C \in \mathcal{B}_X$, the set $\{ V \mid \exists U \in C, f^{-1}(V) \subseteq U \}$ is a Cauchy cover of $Y$.
Then $f$ is an embedding.
\end{lem}
\begin{proof}
We will show, by induction on the generation of a cover $E \in \overline{\mathcal{B}_X}$, that $\{ U' \mid \exists U \in E, f^{-1}(U') \subseteq U \}$ is a Cauchy cover of $Y$.

The base case holds by assumption.
The cases \axref{CT} and \axref{CE} are straightforward.
Now, consider \axref{CG}.
Let $C$ be a Cauchy cover of $X$, and let $\{ D_U \}_{U \in C}$ be a collection of Cauchy covers of $X$ such that $E = \{ U \cap V \mid U \in C, V \in D_U \}$.
By the induction hypothesis, $C' = \{ U' \mid \exists U \in C, f^{-1}(U') \subseteq U \}$ is a Cauchy cover of $Y$.

For each $U' \in C'$, define
\[ D'_{U'} = \{ V' \mid \exists U \in C, \exists V \in D_U, f^{-1}(U') \subseteq U, f^{-1}(V') \subseteq V \}. \]
Then $D'_{U'}$ is a Cauchy cover of $Y$.
Since $\{ U' \cap V' \mid U' \in C', V' \in D'_{U'} \}$ refines $\{ W' \mid \exists W \in E, f^{-1}(W') \subseteq W \}$, the proof is complete.
\end{proof}

\begin{lem}[prod-embedding]
If $f : X \to Y$ and $g : X' \to Y'$ are embeddings, then $f \times g : X \times X' \to Y \times Y'$ is also an embedding.
\end{lem}
\begin{proof}
The precover space structure on $X \times X'$ is generated by covers of the form $\{ \pi_1^{-1}(U) \mid U \in C \}$ and $\{ \pi_2^{-1}(U') \mid U' \in C' \}$,
where $C$ is a Cauchy cover of $X$, and $C'$ is a Cauchy cover of $X'$.

Since $f$ and $g$ are embeddings, the sets $\{ \pi_1^{-1}(V) \mid \exists U \in C, f^{-1}(V) \subseteq U \}$
and $\{ \pi_2^{-1}(V') \mid \exists U' \in C', f^{-1}(V') \subseteq U' \}$ are Cauchy covers on $Y$ and $Y'$, respectively.
These covers refine $\{ V \mid \exists U \in C, (f \times g)^{-1}(V) \subseteq \pi_1^{-1}(U) \}$
and $\{ V' \mid \exists U' \in C', (f \times g)^{-1}(V') \subseteq \pi_2^{-1}(U') \}$.
By the previous lemma, $f \times g$ is an embedding.
\end{proof}

\section{Cauchy filters}
\label{sec:filters}

In the next section, we will define complete cover spaces.
To do this, we first introduce the notion of a Cauchy filter, which we define and study in this section.
Additionally, we will show that Cauchy filters on $\mathbb{Q}$ correspond to Dedekind real numbers.

\begin{defn}
A \emph{filter} on a set $X$ is an upward-closed set of subsets of $X$ that is closed under finite intersections.
A \emph{proper filter} is a filter consisting of inhabited sets.
A filter on a cover space satisfies the \emph{Cauchy condition} if it intersects with every Cauchy cover.
A \emph{Cauchy filter} on a cover space $X$ is a proper filter that satisfies the Cauchy condition.
\end{defn}

\begin{example}
For every point $x$ in a cover space $X$, the set of neighborhoods of $x$ forms a Cauchy filter on $X$.
This filter is denoted by $x^\wedge$.
\end{example}

If a cover space is equipped with a subbase, it becomes easier to determine when a proper filter is Cauchy:

\begin{prop}[cauchy-filter]
If $\mathcal{B}_X$ is a subbase for a cover space $X$, then a proper filter $F$ is Cauchy if and only if it intersects with every cover in $\mathcal{B}_X$.
\end{prop}
\begin{proof}
This follows directly from \rlem{closure-filter}.
\end{proof}

\begin{cor}[metric-cauchy-filter]
If $X$ is a metric space, then a proper filter is Cauchy if and only if it contains an open ball of radius $\varepsilon$ for every $\varepsilon > 0$.
\end{cor}

The notion of a Cauchy filter makes every cover space into a Cauchy space \cite[Definition~1.3.1]{cauchy-spaces}.
However, we will not use this abstraction; instead, we provide the necessary constructions and definitions directly for cover spaces.

Given a function $f : X \to Y$ and a filter $F$ on $X$, we write $f(F)$ for the filter $\{ V \mid f^{-1}(V) \in F \}$ on $Y$.
We say that $f$ is a \emph{Cauchy map} if it preserves Cauchy filters.
Clearly, every cover map is Cauchy.
It is also straightforward to verify that Cauchy maps are continuous:

\begin{prop}[cauchy-continuous]
Every Cauchy map $f : X \to Y$ is a continuous function.
\end{prop}
\begin{proof}
Let $U$ be an open subset of $Y$, and let $x \in X$ be such that $f(x) \in U$.
Since $f(x^\wedge)$ is a Cauchy filter and $\{ f(x) \} \rb U$, there exists a set $V \in f(x^\wedge)$ such that $f(x) \in V \implies V \subseteq U$.
The first condition implies that $f^{-1}(V)$ is a neighborhood of $x$, and the second condition implies that $f^{-1}(V) \subseteq f^{-1}(U)$.
Thus, $f^{-1}(U)$ is also a neighborhood of $x$.
\end{proof}

There is an equivalence relation on Cauchy filters.
Two Cauchy fitlers are \emph{equivalent} if every Cauchy cover contains an element from both filters.
Equivalently, Cauchy filters are equivalent if and only if their intersection is a Cauchy filter.
In particular, if one Cauchy filter is a subset of another, they are equivalent.

Now, we prove a key lemma about Cauchy filters:

\begin{lem}[cauchy-filter-rb]
Let $X$ be a cover space, $F$ and $G$ be equivalent Cauchy filters on $X$, and $U$ and $V$ be subsets of $X$ such that $V \rb U$.
If $V \in F$, then $U \in G$.
\end{lem}
\begin{proof}
Since $V \rb U$, there exists a set $W \in F \cap G$ such that $\overlap{W}{V} \implies W \subseteq U$.
As $W,V \in F$, their intersection also belongs to $F$, hence it is inhabited.
It follows that $W \subseteq U$.
Thus, $U \in G$
\end{proof}

\begin{prop}[cauchy-filter-equiv]
The equivalence of Cauchy filters is indeed an equivalence relation.
\end{prop}
\begin{proof}
The only non-trivial part is transitivity.
Let $F$, $G$, and $H$ be Cauchy filters such that $F$ is equivalent to $G$ and $G$ is equivalent to $H$.
Let $C$ be a Cauchy cover.
By \axref{CR}, there exist sets $U$ and $V$ such that $V \rb U$, $U \in C$, and $V \in F \cap G$.
By \rlem{cauchy-filter-rb}, $U$ belongs to both $F$ and $H$.
\end{proof}

A Cauchy filter $F$ is called \emph{regular} if, for every $U \in F$, there exists a set $V \in F$ such that $V \rb U$.
Regular Cauchy filters can be used as representatives of their equivalence classes.

\begin{example}
For every point $x \in X$, the neighborhood filter $x^\wedge$ is regular.
This follows from \rlem{rb-point}.
\end{example}

The following propositions show that every Cauchy filter is equivalent to a unique regular one:

\begin{prop}[regular-minimal]
Let $F$ and $G$ be equivalent Cauchy filters.
If $F$ is regular, then $F \subseteq G$.
\end{prop}
\begin{proof}
Let $U \in F$.
Since $F$ is regular, there exists a set $V \in F$ such that $V \rb U$.
By \rlem{cauchy-filter-rb}, $U$ belongs to $G$.
\end{proof}

\begin{prop}[regular-filter-repr]
Let $F$ be a Cauchy filter.
The intersection of all Cauchy subfilters of $F$ is a regular Cauchy filter.
It is the unique regular filter equivalent to $F$.
\end{prop}
\begin{proof}
We denote the intersection of all Cauchy subfilters of $F$ by $M$.

First, note that the intersection of an inhabited set of proper filters is a proper filter.
To see that $M$ is a Cauchy filter, let $C$ be a Cauchy cover.
By \axref{CR}, there exist sets $V \in F$ and $U \in C$ such that $V \rb U$.
By \rlem{cauchy-filter-rb}, every Cauchy filter equivalent to $F$ contains $U$.
Thus, $U$ belongs to $M$, showing that $M$ is Cauchy.

Next, we prove that $M$ is regular.
Define $G = \{ U \mid \exists W \in F, \exists V, W \rb V, V \rb U \}$.
By \rprop{rb-props}, $G$ is a proper filter, and by \axref{CR}, it is Cauchy.
Since $G \subseteq F$, we have $M \subseteq G$.
Now, let $U \in M$.
Then there exist sets $W \in F$ and $V$ such that $W \rb V$ and $V \rb U$.
Since $F$ and $M$ are equivalent, \rlem{cauchy-filter-rb} implies $V \in M$.
This shows that $M$ is regular.

Finally, since $M \subseteq F$, it is equivalent to $F$.
Uniqueness follows from \rprop{regular-minimal}.
\end{proof}

Now, we will construct a bijection between regular Cauchy filters and Dedekind reals.
Recall that Dedekind reals are usually defined as Dedekind cuts.
A \emph{Dedekind cut} is a pair $(L,U)$ of subsets of $\mathbb{Q}$ satisfying the following conditions:
\begin{itemize}
\item[(DI)] \axitem{DI} Both $L$ and $U$ are inhabited.
\item[(DL)] \axitem{DL} $a \in L$ if and only if there exists $b > a$ such that $b \in L$.
\item[(DU)] \axitem{DU} $a \in U$ if and only if there exists $b < a$ such that $b \in U$.
\item[(DD)] \axitem{DD} $L$ and $U$ do not intersect.
\item[(DS)] \axitem{DS} If $a < b$, then either $a \in L$ or $b \in U$.
\end{itemize}

Dedekind reals can alternatively be represented by filters.
A filter $F$ on $\mathbb{Q}$ will be called a \emph{Dedekind filter} if it satisfies the following conditions:
\begin{itemize}
\item[(DE)] \axitem{DE} For every rational $\varepsilon > 0$, there exists $a \in \mathbb{Q}$ such that the open interval $(a, a + \varepsilon)$ belongs to $F$.
\item[(DR)] \axitem{DR} If $F$ contains an open interval $(a,d)$, then $F$ also contains an open interval $(b,c)$ such that $a < b < c < d$.
\item[(DT)] \axitem{DT} For every $U \in F$, there exists an open interval $(a,b) \in F$ such that $(a,b) \subseteq U$.
\end{itemize}

\begin{prop}[dedekind-cuts-filters]
There is a bijection between the set of Dedekind cuts and the set of Dedekind filters.
\end{prop}
\begin{proof}
If $(L,U)$ is a Dedekind real, define $F(L,U)$ as the set of all $V$ such that there exist $a \in L$ and $b \in U$ such that $(a,b) \subseteq V$.
Clearly, this is a filter that satisfies \axref{DT}.
Axioms \axref{DL} and \axref{DU} imply \axref{DR}.

Let us prove \axref{DE}.
By \axref{DI}, \axref{DD}, and \axref{DL}, there exist $a \in L$ and $b \in U$ such that $a < b$.
Let $a' = a + \frac{b - a}{3}$ and $b' = b - \frac{b - a}{3}$.
By \axref{DS}, either $a' \in L$ or $b' \in U$.
Thus, the open interval $(a,b)$ can be reduced by a factor of $1.5$.
Repeating this process, we can obtain an arbitrarily small open interval in a finite number of steps, proving \axref{DE}.

Now, let $F$ be a Dedekind filter.
Define:
\[ L_F = \{ a \in \mathbb{Q} \mid \exists b \in \mathbb{Q}, (a,b) \in F \}, \quad U_F = \{ b \in \mathbb{Q} \mid \exists a \in \mathbb{Q}, (a,b) \in F \}. \]
Axiom \axref{DE} implies \axref{DI} and \axref{DS}.
Axiom \axref{DR} implies \axref{DL} and \axref{DU}.
Finally, \axref{DT} implies \axref{DD}.

It is straightforward to verify that $L_{F(L,U)} = L$ and $U_{F(L,U)} = U$.
The fact that $F(L_F,U_F) = F$ follows directly from \axref{DT}.
\end{proof}

The cover space of rational numbers $\mathbb{Q}$ is defined as the one induced by the usual Euclidean metric space structure on $\mathbb{Q}$.
Now, we can show that regular Cauchy filters on this cover space correspond to Dedekind filters:

\begin{prop}[dedekind-cauchy]
A filter on $\mathbb{Q}$ is a Dedekind filter if and only if it is a regular Cauchy filter.
\end{prop}
\begin{proof}
By \rcor{metric-cauchy-filter}, a proper filter on $\mathbb{Q}$ satisfies \axref{DE} if and only if it is a Cauchy filter.
Thus, we need to show that a Cauchy filter is regular if and only if it satisfies \axref{DR} and \axref{DT}.

If a Cauchy filter $F$ satisfies \axref{DR} and \axref{DT}, then, for every $U \in F$, there exist $a < b < c < d$ such that $(b,c) \subseteq (a,d) \subseteq U$ and $(b,c) \in F$.
Since $(b,c) \rb (a,d)$, this shows that $F$ is regular.

Conversely, suppose that $F$ is regular.
Define
\[ F' = \{ U \in F \mid \exists (a,b) \in F, (a,b) \subseteq U \}. \]
By \rcor{metric-cauchy-filter}, $F'$ is a Cauchy filter.
Since $F' \subseteq F$ and $F$ is regular, \rprop{regular-minimal} implies that $F = F'$.
Thus, $F$ satisfies \axref{DT}.

Finally, let us prove that $F$ satisfies \axref{DR}.
Since $F$ is regular, if we have $(a,b) \in F$, then there exists a set $U \in F$ such that $U \rb (a,b)$.
By \axref{DT}, we also have an open interval $(c,d) \in F$ such that $(c,d) \subseteq U$.
\rprop{rb-closure} implies that $a < c$ and $d < b$.
This concludes the proof.

To show \axref{DR}, suppose $(a,b) \in F$.
Because $F$ is regular, there exists $U \in F$ such that $U \rb (a,b)$.
By \axref{DT}, we can find an open interval $(c,d) \in F$ such that $(c,d) \subseteq U$.
Using \rprop{rb-closure}, we conclude that $a < c < d < b$, verifying \axref{DR}.
This completes the proof.
\end{proof}

\section{Completion}
\label{sec:completion}

In this section, we define separated and complete cover spaces and prove that complete cover spaces form a reflective subcategory of $\cat{Cov}$.
Separated cover spaces are analogous to $T_0$-spaces or Hausdorff spaces.
These two notions are equivalent for cover spaces, as we assume regularity.
We will also prove that a certain subcategory of complete cover spaces and Cauchy maps is equivalent to a subcategory of regular topological spaces.

To define separated spaces, we first introduce an equivalence relation on points of a cover space.
Two points $x$ and $y$ are said to be \emph{equivalent} if every Cauchy cover contains a set that includes both $x$ and $y$.
The following lemma provides a useful characterization of this relation, confirming that it is indeed an equivalence relation:

\begin{lem}[separated-char]
Let $x,y$ be a pair of points in a cover space $X$.
Then the following conditions are equivalent:
\begin{enumerate}
\item \label{sc:sub} $x^\wedge \subseteq y^\wedge$.
\item $x^\wedge$ and $y^\wedge$ are equivalent as Cauchy filters.
\item $x^\wedge = y^\wedge$.
\item \label{sc:con} Every neighborhood of $x$ contains $y$.
\item \label{sc:int} Every neighborhood of $x$ intersects every neighborhood of $y$.
\item \label{sc:neighbor} Every Cauchy cover contains a neighborhood of both $x$ and $y$.
\item \label{sc:cov} $x$ and $y$ are equivalent.
\end{enumerate}
\end{lem}
\begin{proof}
The equivalence of the first three conditions follows from \rprop{regular-minimal}, since neighborhood filters are regular.
It is also clear that \eqref{sc:sub} implies \eqref{sc:con}, \eqref{sc:con} implies \eqref{sc:int}, and \eqref{sc:neighbor} implies \eqref{sc:cov}.
Additionally, \eqref{sc:cov} implies \eqref{sc:neighbor} by \axref{CR}.

Now, we prove that \eqref{sc:int} implies \eqref{sc:neighbor}.
Let $C$ be a Cauchy cover.
Then, there exist sets $U,W$ such that $U$ is a neighborhood of $x$, $U \rb W$, and $W \in C$.
Since $U \rb W$, there exists a neighborhood $V$ of $y$ such that $\overlap{V}{U} \implies V \subseteq W$.
By \eqref{sc:int}, $U$ and $V$ intersect, which implies that $V \subseteq W$.
Thus, $W$ is a neighborhood of both $x$ and $y$.

Finally, we prove that \eqref{sc:neighbor} implies \eqref{sc:sub}.
Let $U$ be a neighborhood of $x$.
By \eqref{sc:cov}, there exists a neighborhood $V$ of both $x$ and $y$ such that $x \in V \implies V \subseteq U$
Thus, $U$ is also a neighborhood of $y$.
\end{proof}

\begin{defn}
A cover space $X$ is called \emph{separated} if any two points of $X$ that satisfy the equivalent conditions of \rlem{separated-char} are equal.
\end{defn}

\begin{example}
A discrete cover space is separated.
Indeed, since the cover consisting of singleton sets is Cauchy, any two equivalent points must be equal.
\end{example}

\begin{example}
The cover space induced by a metric space is separated.
Indeed, if two points $x$ and $y$ are equivalent, there exists an open ball of any given radius that contains both $x$ and $y$.
This implies that the distance between $x$ and $y$ is zero, and hence, $x = y$.
\end{example}

A topological space $X$ is \emph{Hausdorff} if, for every pair of points $x,y \in X$, whenever every neighborhood of $x$ intersects every neighborhood of $y$, it follows that $x = y$.
\rlem{separated-char} implies that a cover space is separated if and only if its underlying topological space is Hausdorff.

A subset of a topological space is \emph{dense} if it intersects every inhabited open set.
A function is \emph{dense} if its image is dense.
We include a proof of the following standard fact to demonstrate that it is constructive:

\begin{prop}[dense-unique]
Let $Y$ be a Hausdorff topological space, $S$ a dense subset of a topological space $X$, and $f,g : X \to Y$ continuous functions.
If $f$ and $g$ are equal on every point in $S$, then $f = g$.
\end{prop}
\begin{proof}
Let $x$ be a point in $X$.
Since $Y$ is Hausdorff, to prove that $f(x) = g(x)$, we need to show that every neighborhood of $f(x)$ intersects every neighborhood of $g(x)$.
Let $U$ be a neighborhood of $f(x)$ and $V$ a neighborhood of $g(x)$.
Then $f^{-1}(U) \cap g^{-1}(V)$ is a neighborhood of $x$.
Since $S$ is dense, this neighborhood contains a point $x' \in S$.
Thus, $U$ and $V$ both contain $f(x') = g(x')$.
\end{proof}

Now, we are ready to define complete cover spaces:

\begin{defn}
A cover space is said to be \emph{complete} if it is separated and every Cauchy filter is equivalent to the neighborhood filter of some point.
\end{defn}

\begin{remark}
A cover space is separated if and only if the mapping $(-)^\wedge$, which sends points to regular Cauchy filters, is injective.
Similarly, a cover space is complete if and only if this mapping is bijective.
\end{remark}

In a complete cover space, every Cauchy filter $F$ determines a point, denoted by $F^\vee$, which is the unique point such that $F^{\vee \wedge} \subseteq F$.
The following lemma provides a useful characterization of the neighborhood filter $F^{\vee \wedge}$:

\begin{lem}[filter-point-char]
Let $F$ be a Cauchy filter in a complete cover space.
Then
\[ F^{\vee \wedge} = \{ U \mid \exists V \in F, V \rb U \}. \]
\end{lem}
\begin{proof}
Since $F^{\vee \wedge}$ and $F$ are equivalent, if $V \in F$ and $V \rb U$, then $U \in F^{\vee \wedge}$ by \rlem{cauchy-filter-rb}.
Conversely, suppose $U \in F^{\vee \wedge}$.
By \rlem{rb-point}, there exists a neighborhood $V$ of $F^\vee$ such that $V \rb U$.
Since $F^{\vee \wedge} \subseteq F$, we have $V \in F$, completing the proof.
\end{proof}

Now, we will show that cover maps to complete cover spaces extend uniquely along dense embeddings.
First, we prove a lemma that shows how to lift Cauchy filters along dense embeddings:

\begin{lem}[filter-lift]
Let $f : X \to Y$ be a dense embedding between cover spaces, and let $F$ be a Cauchy filter on $Y$.
Then the set
\[ G = \{ U \mid \exists V,V' \subseteq Y, f^{-1}(V) \subseteq U, V' \rb V, V' \in F \} \]
is a Cauchy filter on $X$ such that $f(G)$ is equivalent to $F$.
\end{lem}
\begin{proof}
Clearly, $G$ is a filter.
To see that it is proper, let $U \in G$.
By definition of $G$, there exist sets $V, V' \subseteq Y$ such that $f^{-1}(V) \subseteq U$, $V' \rb V$, and $V' \in F$.
Since $F$ is a proper filter, there exists a point $y \in V'$.
Since $f$ is dense, there is a point $x \in X$ such that $f(x) \in V$, implying $x \in U$.
Therefore, $G$ is proper.

We now demonstrate that $G$ is Cauchy.
Let $C$ be a Cauchy cover of $X$.
Since $f$ is an embedding, the set
\[ \{ V' \mid \exists V, V' \rb V, \exists U \in C, f^{-1}(V) \subseteq U \} \]
forms a Cauchy cover of $Y$.
Since $F$ is Cauchy, there exist sets $V',V,U$ such that $V' \in F$, $V' \rb V$, $U \in C$, and $f^{-1}(V) \subseteq U$.
This implies $U \in C \cap G$, proving that $G$ is Cauchy.

Finally, we prove that $f(G)$ is equivalent to $F$.
Let $C$ be a Cauchy cover of $Y$.
Since $F$ is Cauchy, there exist sets $V' \in F$ and $V \in C$ such that $V' \rb V$.
Since $f^{-1}(V)$ belongs to $G$ by construction, we have $V \in C \cap F \cap f(G)$.
Therefore, $f(G)$ is indeed equivalent to $F$.
\end{proof}

This lemma implies a useful criterion for completeness.
Given a dense embedding $f : X \to Y$, to prove that $Y$ is complete, it suffices to consider only Cauchy filters that come from $X$:

\begin{lem}[complete-part]
Let $X$ be a cover space, $Y$ a separated cover space, and $f : X \to Y$ a dense embedding.
Then $Y$ is complete if and only if, for every regular Cauchy filter $F$ on $X$, there exists a point $y \in Y$ such that $f(F)$ is equivalent to the neighborhood filter of $y$.
\end{lem}
\begin{proof}
The ``only if'' direction is obvious.
The coverse follows from \rlem{filter-lift}.
\end{proof}

Now, we are ready to prove the extension property:

\begin{thm}[dense-lift]
For every complete cover space $Z$, every dense embedding $f : X \to Y$ between cover spaces, and every Cauchy map $g : X \to Z$,
there exists a unique Cauchy map $\widetilde{g} : Y \to Z$ such that $\widetilde{g} \circ f = g$.
If $g$ is a cover map, then so is $\widetilde{g}$.
\end{thm}
\begin{proof}
The uniqueness follows from \rprop{dense-unique} and \rprop{cauchy-continuous}.
Let us prove the existence.
For each point $y \in Y$, by \rlem{filter-lift}, there exists a Cauchy filter $G_y$ on $X$ such that $f(G_y)$ is equivalent to $y^\wedge$.
We define $\widetilde{g}(y)$ as $g(G_y)^\vee$.

Next, we show that $\widetilde{g}$ is a Cauchy map.
Let $F$ be a Cauchy filter on $Y$ and $C$ a Cauchy cover on $Z$.
We need to show that there exists a set $U \in C$ such that $\widetilde{g}^{-1}(U) \in F$.
Let $G_F$ be the filter constructed for $F$ via \rlem{filter-lift}.
Since $g$ is Cauchy, the filter $g(G_F)$ is Cauchy, so there exist sets $U' \in g(G_F)$ and $U \in C$ such that $U' \rb U$.

From the definition of $G_F$, we get sets $V' \in F$ and $V$ such that $f^{-1}(V) \subseteq g^{-1}(U')$ and $V' \rb V$.
Since $F$ is a filter, it suffices to show that $V' \subseteq \widetilde{g}^{-1}(U)$.
Let $y \in V'$.
We need to show that $U$ is a neighborhood of $\widetilde{g}(y)$.
By \rlem{filter-point-char}, it suffices to show that $U' \in g(G_y)$.
This follows from the definition of $G_y$ and \rlem{rb-point}, as $f^{-1}(V) \subseteq g^{-1}(U')$ and $V$ is a neighborhood of $y$.

Now, we prove that $\widetilde{g} \circ f = g$.
Let $x \in X$.
By \rlem{separated-char}, it suffices to show that every neighborhood of $\widetilde{g}(f(x))$ contains $g(x)$.
Let $W$ be a neighborhood of $\widetilde{g}(f(x))$.
By \rlem{filter-point-char}, there exists a set $W' \in g(G_{f(x)})$ such that $W' \rb W$.
From the definition of $G_{f(x)}$, we get sets $V$ and $V'$ such that $f^{-1}(V) \subseteq g^{-1}(W')$, $V' \rb V$, and $\{ f(x) \} \rb V'$.
This implies that $W$ is a neighborhood of $g(x)$.

Finally, suppose that $g$ is a cover map.
We need to show that $\widetilde{g}$ is also a cover map.
Let $D$ be a Cauchy cover on $Z$.
By \axref{CR}, the set $\{ g^{-1}(W') \mid \exists W \in D, W' \rb W \}$ is a Cauchy cover on $X$.
Since $f$ is an embedding, the set
\[ \{ V'' \mid \exists V,V',W',W, V'' \rb V', V' \rb V, W \in D, W' \rb W, f^{-1}(V) \subseteq g^{-1}(W') \} \]
is a Cauchy cover on $Y$.
We will show that this cover refines $\{ \widetilde{g}^{-1}(W) \mid W \in D \}$.

Let $V,V',V'',W,W'$ be sets such that $V'' \rb V'$, $V' \rb V$, $W \in D$, $W' \rb W$, and $f^{-1}(V) \subseteq g^{-1}(W')$.
Then $V'' \subseteq \widetilde{g}^{-1}(W)$.
To see this, let $y \in V''$.
We need to show that $W$ is a neighborhood of $\widetilde{g}(y) = g(G_y)^\vee$, but this follows from \rlem{filter-point-char} and the definition of $G_y$.
\end{proof}

A \emph{completion} of a cover space $X$ is a complete cover space $Y$ together with a dense embedding $X \to Y$.
By \rthm{dense-lift}, completion is unique up to isomorphism.
We now show that every cover space has a completion:

\begin{thm}[completion]
For every cover space $X$, there exists a complete cover space $C(X)$ and a dense embedding $\eta_X : X \to C(X)$.
\end{thm}
\begin{proof}
Define $C(X)$ as the set of regular Cauchy filters on $X$.
For every subset $U \subseteq X$, let $\widetilde{U} \subseteq C(X)$ denote the set of regular Cauchy filters containing $U$.
The Cauchy covers of $C(X)$ are defined as those covers $C$ refined by $\{ \widetilde{U'} \mid U' \in C' \}$ for some Cauchy cover $C'$ in $X$.
To see that $C$ is a cover, let $F$ be a regular Cauchy filter.
Then there exists a set $U \in F \cap C'$, which implies that we have $V \in C$ such that $F \in \widetilde{U} \subseteq V$.

We first verify that this definition satisfies the axioms of a cover space:
\begin{itemize}
\item Axioms \axref{CT} and \axref{CE}: These follow directly from the definition of $C(X)$.
\item Axiom \axref{CG}: Let $C$ and $\{ D_U \}_{U \in C}$ be Cauchy covers of $C(X)$.
        We need to show that $E = \{ U \cap V \mid U \in C, V \in D_U \}$ is a Cauchy cover.
        For each $U' \in C'$, define a Cauchy cover of $X$:
            \[ D'_{U'} = \{ V' \mid \exists U \in C, V \in D_U, \widetilde{U'} \subseteq U, \widetilde{V'} \subseteq V \}. \]
        By \axref{CG}, the set
            \[ E' = \{ U' \cap V' \mid U' \in C', \exists U \in C, V \in D_U, \widetilde{U'} \subseteq U, \widetilde{V'} \subseteq V \} \]
        is also a Cauchy cover of $X$.
        Since $E$ is refined by $\{ \widetilde{W'} \mid W' \in E' \}$, it is a Cauchy cover of $C(X)$.
\item Axiom \axref{CR}: First, we show that $\widetilde{V} \rb_{C(X)} \widetilde{U}$ whenever $V \rb_X U$.
        If $V \rb_X U$, then the following set is a Cauchy cover of $C(X)$:
            \[ R_{V,U} = \{ \widetilde{W} \mid \overlap{W}{V} \implies W \subseteq U \}. \]
        Let $\widetilde{W} \in R_{V,U}$.
        We will show that $\widetilde{W}$ belongs to the following set:
            \[ \widetilde{R}_{V,U} = \{ W \mid \overlap{W}{\widetilde{V}} \implies W \subseteq \widetilde{U} \}. \]
        If $\widetilde{W}$ intersects $\widetilde{V}$, then there exists a Cauchy filter $F$ containing $W \cap V$.
        Since $F$ is proper, $W$ intersects $V$, which implies $W \subseteq U$.
        It follows that $\widetilde{W} \subseteq \widetilde{U}$.
        Thus, $\widetilde{W} \in \widetilde{R}_{V,U}$.
        This shows that $\widetilde{R}_{V,U}$ is a Cauchy cover of $C(X)$.
        Therefore, $\widetilde{V} \rb_{C(X)} \widetilde{U}$.

        Now, we proceed to verify \axref{CR}.
        Let $C$ be a Cauchy cover of $C(X)$.
        From the previous argument, we know that $\{ \widetilde{V'} \mid \exists U' \in C', V' \rb_X U' \}$ refines $\{ V \mid \exists U \in C, V \rb_{C(X)} U \}$.
        Since the former is a Cauchy cover, the latter is also a Cauchy cover of $C(X)$.
        This completes the proof of \axref{CR}.
\end{itemize}

Define $\eta_X(x) = x^\wedge$ for all $x \in X$.
We now check that $\eta_X : X \to C(X)$ is a dense embedding:
\begin{itemize}
\item $\eta_X$ is a cover map: Let $C'$ be a Cauchy cover of $X$.
The cover $\{ \eta_X^{-1}(\widetilde{U'}) \mid U' \in C' \}$ is refined by $\{ V' \mid \exists U' \in C', V' \rb_X U' \}$, which is Cauchy by \axref{CR}.
This shows that the former cover is also Cauchy.
\item $\eta_X$ is an embedding: This follows directly from the definition of $C(X)$.
\item $\eta_X$ is dense:
        Let $F$ be a regular Cauchy filter on $X$, and let $V$ be a neighborhood of $F$ in $C(X)$.
        Since $F$ is proper, it suffices to show that $\eta_X^{-1}(V) \in F$.

        Note that $\{ F \} \rb_{C(X)} V$ implies the existence of a Cauchy cover $C'$ of $X$ such that $\{ \widetilde{U'} \mid U' \in C' \}$ refines $\{ W \mid F \in W \implies W \subseteq V \}$.
        Since $F$ is a Cauchy filter and $\{ V' \mid \exists U' \in C', V' \rb_X U' \}$ is a Cauchy cover, there exist $V' \in F$ and $U' \in C'$ such that $V' \rb_X U'$.
        Then $U' \in F$, and $F \in \widetilde{U'} \subseteq V$.

        Now, it suffices to show that $V' \subseteq \{ x \in X \mid \eta_X(x) \in V \}$.
        For $x \in V'$, the neighborhood $U'$ of $x$ satisfies $\eta_X(x) = x^\wedge \in \widetilde{U'}$.
        It follows that $\eta_X(x) \in V$, which completes the proof that $\eta_X$ is dense.
\end{itemize}

To finish the proof, we verify the following properties of $C(X)$:
\begin{itemize}
\item $C(X)$ is separated:
        Let $F$ and $G$ be regular Cauchy filters on $X$ such that every Cauchy cover in $C(X)$ contains a set which contains both $F$ and $G$.
        This implies that, for every Cauchy cover $C'$ in $X$, there exists a set $U' \in C'$ such that $F,G \in \widetilde{U'}$.
        This means that $F$ and $G$ are equivalent Cauchy filters and.
        Since $F$ and $G$ are regular, this proves that $F = G$.
\item $C(X)$ is complete:
        By \rlem{complete-part}, it suffices to show that for any regular Cauchy filter $F$ on $X$, the Cauchy filters $\eta_X(F)$ and $F^\wedge$ are equivalent.
        We just proved that $\eta_X^{-1}(V) \in F$ for every neighborhood $V$ of $F$.
        This means that $F^\wedge$ is a subset of $\eta_X(F)$, so they are equivalent.
\end{itemize}
\end{proof}

\begin{cor}[complete-reflective]
The category of complete cover spaces forms a reflective subcategory of $\cat{Cov}$.
\end{cor}

The embedding $\eta_X$ might not be injective in general, but it is for separated cover spaces.
In fact, this property characterizes separated spaces:

\begin{prop}[embedding-injective]
Given a cover space $X$, the following conditions are equivalent:
\begin{enumerate}
\item \label{ei:sep} $X$ is a separated cover space.
\item \label{ei:any} Every embedding $X \to Y$ is injective.
\item \label{ei:eta} The embedding $\eta_X : X \to C(X)$ is injective.
\item \label{ei:ex} There exists an injective Cauchy map $X \to Y$ into a separated cover space $Y$.
\end{enumerate}
\end{prop}
\begin{proof}
We prove the equivalences step by step:

\paragraph{\eqref{ei:sep} $\iff$ \eqref{ei:eta}:} This follows from \rlem{separated-char}.

\paragraph{\eqref{ei:any} $\implies$ \eqref{ei:eta} and \eqref{ei:eta} $\implies$ \eqref{ei:ex}:} These are obvious.

\paragraph{\eqref{ei:ex} $\implies$ \eqref{ei:eta}:}
Suppose there exists an injective Cauchy map $f : X \to Y$ into a separated cover space $Y$.
Since $Y$ is separated, the embedding $\eta_Y : Y \to C(Y)$ is also injective.
By the naturality of $\eta$, we have that $\eta_Y \circ f = C(f) \circ \eta_X$.
Since both $f$ and $\eta_Y$ are injective, it follows that $\eta_X$ is injective.

\paragraph{\eqref{ei:sep} $\implies$ \eqref{ei:any}:}
Assume $X$ is a separated cover space.
Let $f : X \to Y$ be an embedding, and suppose $x,x' \in X$ satisfy $f(x) = f(x')$.
We must show $x = x'$.

Since $X$ is separated, it suffices to verify that every neighborhood $U$ of $x$ contains $x'$.
From $\{x\} \rb_X U$ and the fact that $f$ is an embedding, the set 
\[
\{ V \mid \exists W \subseteq X, (x \in W \implies W \subseteq U) \text{ and } f^{-1}(V) \subseteq W \}
\]
is a Cauchy cover of $Y$.
In particular, this is a cover, so there exist sets $V$ and $W$ such that $f(x) = f(x') \in V$, $f^{-1}(V) \subseteq W$, and $x \in W \implies W \subseteq U$.

The first two conditions imply $x, x' \in W$, while the last condition implies $W \subseteq U$.
Thus, $x' \in U$, completing the proof that $x = x'$.
\end{proof}

\section{Strongly regular spaces}
\label{sec:strongly-regular}

In this section, we introduce an alternative regularity condition for cover spaces.
Classically, this condition is equivalent to \axref{CR}, but constructively, it is stronger.
Additionally, there exists a stronger notion of completeness for strongly regular cover spaces.
We will establish a close relationship between strongly complete strongly regular cover spaces and strongly regular sober topological spaces.

Let $U$ and $V$ be open subsets of a topological space $X$. We say that $V$ is \emph{strongly rather below} $U$, denoted $V \rb^s U$, if $X = \mathrm{int}(X \backslash V) \cup U$.
A topological space is said to be \emph{strongly regular} if every open set $U$ can be expressed as the union of open sets that are strongly rather below $U$.

The notion of strong regularity extends naturally to precover spaces.
For subsets $V$ and $U$ of a precover space $X$, we say that $V$ is \emph{strongly rather below} $U$, written $V \rb^s U$, if the collection $\{ X \backslash V, U \}$ forms a Cauchy cover.
This relation satisfies properties analogous to those in \rprop{rb-props}:

\begin{prop}[rbs-props]
Let $U,V,U',V'$ be subsets of a precover space $X$. The following properties hold:
\begin{enumerate}
\item If $V \rb^s U$, then $V \rb U$.
\item If $V' \subseteq V \rb^s U \subseteq U'$, then $V' \rb^s U'$.
\item If $V \rb^s U$ and $V' \rb^s U'$, then $V \cap V' \rb^s U \cap U'$.
\item $V \rb^s X$.
\item $\varnothing \rb^s U$.
\end{enumerate}
\end{prop}

A precover space is called \emph{strongly regular} if, for every Cauchy cover $C$, the set $\{ V \mid \exists U \in C, V \rb^s U \}$ is a Cauchy cover.
Since $V \rb^s U$ implies $V \rb U$, every strongly regular precover space is a cover space.

\begin{example}
Recall that if $X$ is a topological space, then $T(X)$ is a precover space where a cover is Cauchy if it contains a neighborhood of every point.
Clearly, if $X$ is a strongly regular topological space, then $T(X)$ is a strongly regular cover space.
\end{example}

\begin{example}
Every metric space is strongly regular as a cover space.
\end{example}

\begin{example}
More generally, a uniform space $X$ is said to be \emph{strongly regular} if it satisfies the following strengthened version of \axref{UU}:
For every $C \in \mathcal{U}_X$, there exists $D \in \mathcal{U}_X$ such that, for every $V \in D$,
there exists $U \in C$ satisfying the following: for every $W \in D$, either $W$ does not intersect $V$, or $W$ is a subset of $U$.

With this definition, every metric space is a strongly regular uniform space, and every strongly regular uniform space is a strongly regular cover space.
\end{example}

\begin{example}
Strongly regular cover spaces are reflective in the category of cover spaces.
The proof is identical to the proof of \rprop{regular-reflective}.
Thus, every cover space can be transformed into a strongly regular cover space.
\end{example}

The following lemma establishes the equivalence of the conditions $\{ x \} \rb U$ and $\{ x \} \rb^s U$ in a strongly regular cover space:

\begin{lem}[rbs-point]
If $U$ is a neighborhood of a point $x$ in a strongly regular cover space, then there exists a neighborhood $V$ of $x$ such that $V \rb^s U$.
\end{lem}
\begin{proof}
The proof is essentially the same as the proof of \rlem{rb-point}.
\end{proof}

The following proposition relates the strong regularity conditions of topological spaces and cover spaces:

\begin{prop}[u-strongly-regular]
If $X$ is a strongly regular cover space, then its underlying topological space is strongly regular.
\end{prop}
\begin{proof}
The proof follows from \rlem{rbs-point} and is analogous to the proof of \rprop{top-regular}.
\end{proof}

We define a topological space $X$ to be \emph{strongly Hausdorff} if, for every pair of points $x,y \in X$, the following holds:
If the intersection of every neighborhood of $x$ with every neighborhood of $y$ is not empty, then $x = y$.

\begin{lem}
Every separated strongly regular cover space is strongly Hausdorff.
\end{lem}
\begin{proof}
Let $x$ and $y$ be points in $X$ such that the intersection of every neighborhood of $x$ with every neighborhood of $y$ is not empty.
We need to show that every Cauchy cover $C$ contains a set that includes both $x$ and $y$. 

By \rlem{rbs-point}, there exist $U \in C$ and a neighborhood $V$ of $x$ such that $V \rb^s U$.
Then either $y \in U$ or the complement of $V$ is a neighborhood of $y$.
The latter case is impossible since this would imply that $V$ is a neighborhood of $x$ that does not intersect a neighborhood of $y$.
Thus, $y \in U$, and the result follows.
\end{proof}

A continuous function with values in a Hausdorff space is determined by its values on a dense subset.
For strongly Hausdorff spaces, a weaker condition on the subset suffices.
A subset $S$ of a topological space is called \emph{weakly dense} if every open set that does not intersect $S$ is empty.
Similarly, a function is said to be \emph{weakly dense} if its image is weakly dense. 

We now prove a version of \rprop{dense-unique} for strongly Hausdorff spaces:

\begin{prop}[weakly-dense-unique]
Let $Y$ be a strongly Hausdorff topological space, $S$ a weakly dense subset of a topological space $X$, and $f,g : X \to Y$ continuous functions.
If $f$ and $g$ agree on every point of $S$, then $f = g$.
\end{prop}
\begin{proof}
Let $x \in X$.
To show that $f(x) = g(x)$, let $U$ be a neighborhood of $f(x)$ and $V$ a neighborhood of $g(x)$.
We need to prove that $U \cap V \neq \varnothing$. 

Assume, for contradiction, that $U \cap V = \varnothing$.
Then $f^{-1}(U) \cap g^{-1}(V)$ is a neighborhood of $x$.
Since $S$ is weakly dense, we can assume that this neighborhood contains a point $x' \in S$.
Thus, $U$ and $V$ both contain $f(x') = g(x')$, which contradicts $U \cap V = \varnothing$.
\end{proof}

To define strongly complete spaces, we first need to introduce the concept of weakly proper filters.

\begin{defn}
A \emph{weakly proper filter} is a filter that does not contain the empty set.
A \emph{weakly Cauchy filter} on a cover space $X$ is a weakly proper filter that intersects every Cauchy cover.
\end{defn}

\begin{remark}
Classically, weakly proper filters and weakly Cauchy filters coincide with proper filters and Cauchy filters, respectively.
Constructively, however, these conditions are generally weaker.
Nevertheless, as we will see in Section~\ref{sec:reals}, there are cases where weakly Cauchy filters coincide with Cauchy filters even in the constructive setting.
\end{remark}

We define a function between cover spaces as a \emph{strongly Cauchy map} if it preserves weakly Cauchy filters.
Clearly, every cover map is strongly Cauchy, and every strongly Cauchy map is Cauchy.

The equivalence of weakly Cauchy filters is defined as before:
Two weakly Cauchy fitlers are \emph{equivalent} if every Cauchy cover contains an element from both filters.
Equivalently, weakly Cauchy filters are equivalent if and only if their intersection is a weakly Cauchy filter.
In particular, if one weakly Cauchy filter is a subset of another, they are equivalent.

We now state an analogue of \rlem{cauchy-filter-rb}:

\begin{lem}[weakly-cauchy-filter-rb]
Let $X$ be a cover space, and let $F$ and $G$ be equivalent weakly Cauchy filters on $X$.
If $U$ and $V$ are subsets of $X$ such that $V \rb^s U$, and $V \in F$, then $U \in G$.
\end{lem}
\begin{proof}
Since $V \rb^s U$, either $U \in F \cap G$ or $(X \backslash V) \in F \cap G$.
The latter case is impossible because $V \in $F and $\varnothing \notin F$.
Thus, $U \in G$.
\end{proof}

Finally, we confirm that the equivalence of weakly Cauchy filters satisfies the properties of an equivalence relation:

\begin{prop}
The equivalence of Cauchy filters is indeed an equivalence relation.
\end{prop}
\begin{proof}
The proof is identical to the proof of \rprop{cauchy-filter-equiv}.
\end{proof}

A weakly Cauchy filter $F$ is called \emph{strongly regular} if, for every $U \in F$, there exists a set $V \in F$ such that $V \rb^s U$.
Strongly regular weakly Cauchy filters serve as unique representatives of their equivalence classes in the context of strongly regular cover spaces.

\begin{example}
For every point $x$ in a strongly regular cover space, the neighborhood filter $x^\wedge$ is strongly regular.
This follows directly from \rlem{rbs-point}.
\end{example}

The following propositions establish that every weakly Cauchy filter is equivalent to a unique strongly regular one:

\begin{prop}[strongly-regular-minimal]
Let $F$ and $G$ be equivalent weakly Cauchy filters.
If $F$ is strongly regular, then $F \subseteq G$.
\end{prop}
\begin{proof}
The proof is analogous to that of \rprop{regular-minimal}.
\end{proof}

\begin{prop}
Let $F$ be a weakly Cauchy filter.
The intersection of all weakly Cauchy subfilters of $F$ is a strongly regular weakly Cauchy filter.
This is the unique strongly regular filter equivalent to $F$.
\end{prop}
\begin{proof}
The proof is analogous to that of \rprop{regular-filter-repr}.
\end{proof}

Now, we can discuss strongly complete cover spaces.

\begin{defn}
A strongly regular cover space is called \emph{strongly complete} if it is separated and every weakly Cauchy filter is equivalent to the neighborhood filter of some point.
\end{defn}

The following lemmas are analogous to those for regular spaces:

\begin{lem}[weak-filter-point-char]
Let $F$ be a weakly Cauchy filter in a strongly complete cover space.
Then $F^{\vee \wedge} = \{ U \mid \exists V \in F, V \rb^s U \}$.
\end{lem}
\begin{proof}
The proof is analogous to that of \rlem{filter-point-char}.
\end{proof}

\begin{lem}[weak-filter-lift]
Let $f : X \to Y$ be a weakly dense embedding between cover spaces, and let $F$ be a weakly Cauchy filter on $Y$.
Then
\[ G = \{ U \mid \exists V,V' \subseteq Y, f^{-1}(V) \subseteq U, V' \rb V, V' \in F \} \]
is a weakly Cauchy filter on $X$ such that $f(G)$ is equivalent to $F$.
\end{lem}
\begin{proof}
We have already shown in \rlem{filter-lift} that this filter satisfies the Cauchy condition.
It remains to show that $G$ is weakly proper.

Assume, for contradiction, that $\varnothing \in G$.
Then there exist sets $V$ and $V'$ such that $f^{-1}(V) \subseteq \varnothing$, $V' \rb V$, and $V' \in F$.
We will show that $V'$ is empty, which contradicts the fact that $F$ is weakly proper.

Suppose $y \in V'$.
Then $y$ belongs to the interior of $V$.
Since $f$ is weakly dense and the interior of $V$ is inhabited, we can assume that $V$ intersects the image of $f$.
This contradicts the condition $f^{-1}(V) \subseteq \varnothing$.
Thus, $G$ is weakly proper.
\end{proof}

\begin{lem}[strongly-complete-part]
Let $X$ be a strongly regular cover space, $Y$ a separated strongly regular cover space, and $f : X \to Y$ a weakly dense embedding.
Then $Y$ is strongly complete if and only if, for every strongly regular Cauchy filter $F$ on $X$, there exists a point $y \in Y$ such that $f(F)$ is equivalent to the neighborhood filter of $y$.
\end{lem}
\begin{proof}
The ``if'' direction follows from \rlem{weak-filter-lift}.
The converse is straightforward.
\end{proof}

With these lemmas, we can prove an analogue of \rthm{dense-lift}:

\begin{thm}[weakly-dense-lift]
For every cover space $X$, every strongly regular cover space $Y$, every strongly complete cover space $Z$, every weakly dense embedding $f : X \to Y$, and every strongly Cauchy map $g : X \to Z$,
there exists a unique strongly Cauchy map $\widetilde{g} : Y \to Z$ such that $\widetilde{g} \circ f = g$.
Moreover, if $g$ is a cover map, then so is $\widetilde{g}$.
\end{thm}
\begin{proof}
The uniqueness follows from \rprop{weakly-dense-unique} and \rprop{cauchy-continuous}.
The proof of existence is analogous to that of \rthm{dense-lift}, using \rlem{weak-filter-point-char} and \rlem{weak-filter-lift}.
\end{proof}

A \emph{strong completion} of a strongly regular cover space $X$ is a strongly complete cover space $Y$ together with a weakly dense embedding $X \to Y$.
\rthm{weakly-dense-lift} implies that such a completion is unique up to isomorphism.
As before, we can show that every strongly regular cover space has a strong completion:

\begin{thm}[strong-completion]
For every strongly regular cover space $X$, there exists a strongly complete cover space $C(X)$ and a weakly dense embedding $\eta_X : X \to C(X)$.
\end{thm}
\begin{proof}
Define $C(X)$ as the set of strongly regular weakly Cauchy filters on $X$.
For every subset $U \subseteq X$, define $\widetilde{U} \subseteq C(X)$ as the set of all strongly regular weakly Cauchy filters containing $U$.
The Cauchy covers of $C(X)$ are defined as the set of covers $C$ that are refined by $\{ \widetilde{U'} \mid U' \in C' \}$ for some Cauchy cover $C'$ of $X$.
The proof that $C(X)$ forms a precover space is analogous to that in \rthm{completion}.

To prove that $\widetilde{V} \rb^s_{C(X)} \widetilde{U}$ whenever $V \rb^s_X U$, note that $\widetilde{X \backslash V}$ does not intersect $\widetilde{V}$.
Indeed, if a filter $F \in C(X)$ belonged to their intersection, it would contain both $X \backslash V$ and $V$, which is impossible since $F$ is weakly proper.
This implies that $\{ \widetilde{X \backslash V}, \widetilde{U} \}$ refines $\{ C(X) \backslash \widetilde{V}, \widetilde{U} \}$.
Thus, $V \rb^s_X U$ implies $\widetilde{V} \rb^s_{C(X)} \widetilde{U}$, showing that $C(X)$ is strongly regular.

Next, define $\eta_X(x)$ as the neighborhood filter $x^\wedge$.
We can prove that $\eta_X$ is an embedding of cover spaces as in \rthm{completion}.
To show that $\eta_X$ is weakly dense, let $V$ be an open subset of $C(X)$ that does not intersect the image of $\eta_X$.
We need to prove that $V$ is empty.
Suppose $F \in V$.
In \rthm{completion}, we already proved that the set $\{ x \in X \mid \eta_X(x) \in V \}$ belongs to $F$, but this set is empty by assumption, contradicting the weak properness of $F$.
Thus, $V$ must be empty, proving that $\eta_X$ is dense.

Finally, the completeness of $C(X)$ follows from \rlem{strongly-complete-part} as in \rthm{completion}, completing the proof.
\end{proof}

\begin{cor}[strongly-complete-reflective]
The category of strongly complete cover spaces forms a reflective subcategory of the category of strongly regular cover spaces.
\end{cor}

Now, let us discuss the relationship between strongly complete cover spaces and sober topological spaces.
First, recall the definition of sober topological spaces.
Let $X$ be a topological space.

A filter $F$ on $X$ is called \emph{open} if, for every $U \in F$, there exists an open set $V \in F$ such that $V \subseteq U$.
An open filter $F$ is called \emph{completely prime} if, for every family of open sets $\{ U_i \}_{i \in I}$, the condition $\bigcup_{i \in I} U_i \in F$ implies $U_i \in F$ for some $i \in I$.
If $x$ is a point in $X$, then the neighborhood filter $x^\wedge$ is a completely prime open filter.

A topological space is called \emph{sober} if every completely prime open filter is of the form $x^\wedge$ for a unique point $x \in X$.

\begin{prop}[u-sober]
If $X$ is a strongly complete cover space, then its underlying topological space is sober.
\end{prop}
\begin{proof}
Let $F$ be a completely prime open filter.
Clearly, $F$ is also weakly Cauchy.
By \rlem{rbs-point}, for every open set $U$, we have $U = \bigcup \{ V \mid V \rb^s_X U \}$.
It follows that $F$ is strongly regular.
By the completeness of $X$, there exists a unique point $x \in X$ such that $F = x^\wedge$.
\end{proof}

Next, we prove the converse of this proposition.

\begin{prop}[t-complete]
If $X$ is a strongly regular sober space, then $T(X)$ is a strongly complete cover space.
\end{prop}
\begin{proof}
It suffices to show that every strongly regular weakly Cauchy filter is completely prime.
Let $F$ be a strongly regular weakly Cauchy filter, and let $\{ U_i \}_{i \in I}$ be a family of open sets such that $\bigcup_{i \in I} U_i \in F$.
Since $F$ is strongly regular, there exists $V \in F$ such that $V \rb^s \bigcup_{i \in I} U_i$.
This implies that $\{ X \backslash V \} \cup \{ U_i \mid i \in I \}$ is a neighborhood cover of $X$.

Since $F$ is a weakly proper filter and $V \in F$, the set $X \backslash V$ does not belong to $F$.
Thus, there exists some $i \in I$ such that $U_i \in F$, because $F$ is a Cauchy filter.
This completes the proof.
\end{proof}

Now, we can show that the sets of Cauchy maps, strongly Cauchy maps, and continuous maps between strongly complete cover spaces coincide:

\begin{prop}[u-fully-faithful]
Let $X$ be a complete cover space and $Y$ be a cover space.
Then a map $f : X \to Y$ is Cauchy if and only if it is continuous.
If $X$ is strongly complete, then such a map is also strongly Cauchy.
\end{prop}
\begin{proof}
If $f$ is Cauchy, then it is continuous by \rprop{cauchy-continuous}.
Conversely, assume that $f$ is continuous.
Let $F$ be a Cauchy filter on $X$, and let $C$ be a Cauchy cover of $Y$.
We need to show that there exists $U \in C$ such that $f^{-1}(U) \in F$.

Since $X$ is complete, there exists a point $x \in X$ such that $x^\wedge \subseteq F$.
As $C$ is a Cauchy cover, there exists $U \in C$ such that $\{ f(x) \} \rb U$.
Since $f$ is continuous, $f^{-1}(U)$ is a neighborhood of $x$.
Because $x^\wedge \subseteq F$, we conclude that $f^{-1}(U) \in F$.

If $X$ is strongly complete, the same proof shows that $f$ preserves weakly Cauchy filters, making $f$ strongly Cauchy.
\end{proof}

Finally, we can prove that the category of strongly complete cover spaces is equivalent to the category of strongly regular sober spaces:

\begin{thm}[cover-cauchy-top]
The forgetful functor from the category of cover spaces and strongly Cauchy maps to the category of topological spaces restricts to an equivalence between the category of strongly complete cover spaces with strongly Cauchy maps and the category of strongly regular sober spaces.
\end{thm}
\begin{proof}
By \rprop{u-strongly-regular}, \rprop{u-sober}, and \rprop{u-fully-faithful}, we have a functor that is fully faithful.
Additionally, \rprop{t-complete} implies that this functor is essentially surjective on objects.
Thus, the forgetful functor restricts to an equivalence between the two categories.
\end{proof}

\section{Locales}
\label{sec:locales}

In the previous section, we showed that the category of strongly complete cover spaces and strongly Cauchy maps is equivalent to the category of strongly regular sober topological spaces.
In this section, we establish a similar result for the category of strongly complete cover spaces and cover maps instead of strongly Cauchy maps.
To achieve this, we replace the category of sober topological spaces with the category of locales.

Recall that a \emph{frame} is a complete distributive lattice.
A map of frames is a function that preserves finite meets and arbitrary joins.
The category of locales is the opposite of the category of frames.
Given a locale $M$, we denote its underlying frame by $O_M$.
Elements of $O_M$ are called \emph{opens} of $M$.
For a map of locales $f : M \to N$, the corresponding frame map is written as $f^* : O_N \to O_M$.

A \emph{filter} in a lattice $M$ is an upward closed subset of $M$ that is closed under finite meets.
A \emph{completely prime filter} in a complete lattice is a filter $F$ such that the following condition holds:
\[ \bigvee_{j \in J} a_j \in F \implies \exists j \in J, a_j \in F. \]
A \emph{point} of a locale $M$ is a completely prime filter in $O_M$.
The set of points of a locale is denoted by $P(M)$.
There are adjoint functions $\varepsilon^* \dashv \varepsilon_*$ between $O_M$ and the power set of $P(M)$:
\[ \varepsilon^*(a) = \{ F \in P(M) \mid a \in F \}, \quad \varepsilon_*(U) = \bigvee \{ a \in O_M \mid \varepsilon^*(a) \subseteq U \}. \]

The set of points of a locale $M$ can be equipped with the structure of a precover space.
We define $\mathcal{C}_{P(M)}$ as the set of those $C$ for which $\bigvee \{ \varepsilon_*(U) \mid U \in C \} = \top$.

\begin{prop}
For every locale $M$, the pair $(P(M),\mathcal{C}_{P(M)})$ is a precover space.
\end{prop}
\begin{proof}
First, let us verify that every $C$, satisfying $\bigvee \{ \varepsilon_*(U) \mid U \in C \} = \top$, is indeed a cover.
If $F$ is a point of $M$, there exist a set $U \in C$ and an open $a \in F$ such that $\varepsilon^*(a) \subseteq U$.
This implies that $F \in U$.

Condition \axref{CT} is obvious, and \axref{CE} follows from the monotonicity of $\varepsilon_*$.
Condition \axref{CG} is satisfied due to the distributivity of meets over joins and the fact that $\varepsilon_*$ preserves meets.
\end{proof}

The following simple lemma provides a useful characterization of Cauchy covers in $P(M)$:

\begin{lem}[locale-cauchy]
If $M$ is a locale, then a set is a Cauchy cover if and only if it is of the form $\{ \varepsilon^*(a) \mid a \in C \}$ for some set of opens $C \subseteq O_M$ such that $\bigvee C = \top$.
\end{lem}
\begin{proof}
Suppose $\bigvee C = \top$.
Then $C' = \{ \varepsilon^*(a) \mid a \in C \}$ is a Cauchy cover because:
\[ \top = \bigvee C \leq \bigvee \{ \varepsilon_*(U) \mid U \in C' \}. \]
Conversely, let $D$ be a Cauchy cover.
Define $D' = \{ \varepsilon_*(U) \mid U \in D \}$.
Since $\bigvee D' = \top$, it follows that $D = \{ \varepsilon^*(a) \mid a \in D' \}$, completing the proof.
\end{proof}

The following lemma describes the neighborhood relation in $P(M)$ in terms of the original locale $M$:

\begin{lem}[locale-cover-neighborhood]
Let $M$ be a locale, and let $F$ be a completely prime filter in $O_M$.
If $U$ is a neighborhood of $F$ in $P(M)$, then there exists an open $a \in O_M$ such that $\varepsilon^*(a) \subseteq U$.
\end{lem}
\begin{proof}
Suppose that ${ F } \rb U$, which means $\top = \bigvee \{ \varepsilon_*(W) \mid F \in W \implies W \subseteq U \}$.
Let $W$ be a set satisfying $F \in W \implies W \subseteq U$.
If $\varepsilon_*(W) \in F$, then $F \in \varepsilon^*\varepsilon_*(W) \subseteq W$, which implies $\varepsilon^*\varepsilon_*(W) \subseteq W \subseteq U$.

Thus, $\varepsilon_*(W)$ belongs to the set $\{ a \in O_M \mid a \in F \implies \varepsilon^*(a) \subseteq U \}$.
It follows that $\top = \bigvee \{ a \in O_M \mid a \in F \implies \varepsilon^*(a) \subseteq U \}$.
Since $F$ is completely prime, there exists $a \in F$ such that $\varepsilon^*(a) \subseteq U$.
This completes the proof.
\end{proof}

Every map of locales $f : M \to N$ induces a function on points, $P(f) : P(M) \to P(N)$, defined by:
\[ P(f)(F) = \{ a \in O_N \mid f^*(a) \in F \}. \]
It is straightforward to verify that $P(f)$ is a cover map.
Thus, we obtain a functor $P : \cat{Loc} \to \cat{Precov}$.

For every open $b$ in a locale $M$, there exists the largest open $\neg b$ satisfying equation $b \wedge \neg b = \bot$.
This open is given by the formula:
\[ \neg b = \bigvee \{ a \in O_M \mid b \wedge a = \bot \}. \]

\begin{defn}
Given two opens $b,a \in O_M$ in a locale $M$, we say that $b$ is \emph{rather below} $a$, written $b \rb a$, if $\neg b \vee a = \top$.
A locale $M$ is called \emph{regular} if, for every open $a \in O_M$, we have $a = \bigvee \{ b \in O_M \mid b \rb a \}$.
\end{defn}

The following proposition shows that every regular locale gives rise to a strongly regular cover space:

\begin{prop}[locale-regular]
If $M$ is a regular locale, then $P(M)$ is a strongly complete cover space.
\end{prop}
\begin{proof}
\textbf{Step 1: $P(M)$ is strongly regular.}
First, we show that if $a,b \in O_M$ are such that $b \rb a$, then $\varepsilon^*(b) \rb^s \varepsilon^*(a)$.
Suppose $b \rb a$, i.e., $\top \leq \neg b \vee a$.
Then the set $\{ \varepsilon^*(\neg b), \varepsilon^*(a) \}$ is a Cauchy cover.
It suffices to show that $\varepsilon^*(\neg b)$ is a subset of $X \backslash \varepsilon^*(b)$.
This follows because every point of $M$ is a proper filter: if a point contains both $\neg b$ and $b$, it must also contain $\bot$, which is impossible.
Thus, $\varepsilon^*(b) \rb^s \varepsilon^*(a)$, and regularity of $M$ implies that $P(M)$ is strongly regular.

\textbf{Step 2: $P(M)$ is separated.}
Let $F$ and $G$ be a pair of equivalent points in $P(M)$.
By \rlem{separated-char}, every neighborhood $U$ of $F$ contains $G$, and vice versa.
We now show that $F \subseteq G$.

Let $a \in F$.
Since $M$ is regular, $a = \bigvee \{ b \mid b \rb a \}$.
Thus, there exists some $b \in F$ such that $b \rb a$.
This implies $\varepsilon^*(b) \rb \varepsilon^*(a)$, and since $F \in \varepsilon^*(b)$, the set $\varepsilon^*(a)$ is a neighborhood of $F$.
By equivalence, $G \in \varepsilon^*(a)$, which implies $a \in G$.
Similarly, every $a \in G$ belongs to $F$, so $F = G$.
Hence, $P(M)$ is separated.

\textbf{Step 3: $P(M)$ is strongly complete.}
For every weakly Cauchy filter $F$ in $P(M)$, define a point $\widetilde{F}$ in $M$ as follows:
\[
\widetilde{F} = \{ a \in O_M \mid \exists b \rb a, \varepsilon^*(b) \in F \}.
\]
It is straightforward to verify that $\widetilde{F}$ is a filter in $O_M$.

To show that $\widetilde{F}$ is completely prime, suppose $\widetilde{F}$ contains $\bigvee_{j \in J} a_j$.
Then there exists some $b \in O_M$ such that $b \rb \bigvee_{j \in J} a_j$ and $\varepsilon^*(b) \in F$.
The set 
\[
\{ \varepsilon^*(\neg b) \} \cup \{ \varepsilon^*(b') \mid \exists j \in J, \, b' \rb a_j \}
\]
is a Cauchy cover.
If $F$ contains $\varepsilon^*(\neg b)$, then it also contains
$\varepsilon^*(\neg b) \cap \varepsilon^*(b) = \varepsilon^*(\bot) = \varnothing$,
contradicting the fact that $F$ is weakly proper.
Thus, there exists some $b' \in O_M$ such that $\varepsilon^*(b') \in F$ and $b' \rb a_j$ for some $j \in J$.
This implies $a_j \in \widetilde{F}$, proving that $\widetilde{F}$ is completely prime.

Finally, to show that $\widetilde{F}^\wedge$ is equivalent to $F$, consider any neighborhood $U$ of $\widetilde{F}$ in $P(M)$.
By \rlem{locale-cover-neighborhood}, there exist opens $a,b \in O_M$ such that $b \rb a$, $\varepsilon^*(b) \in F$, and $\varepsilon^*(a) \subseteq U$.
This implies $U \in F$, completing the proof that $P(M)$ is strongly complete.
\end{proof}

Another property of locales that we will require is properness.
A locale $M$ is called \emph{proper} if $\varepsilon_*(\varnothing) = \bot$.
This notion also has a corresponding interpretation in the context of precover spaces.
A precover space is called \emph{proper} if $C$ is a Cauchy cover whenever $C \cup \{ \varnothing \}$ is a Cauchy cover.

The following proposition establishes a connection between these two notions:

\begin{prop}[locale-cover-proper]
If a locale $M$ is proper, then the precover space $P(M)$ is also proper.
\end{prop}
\begin{proof}
Assume that $M$ is proper and that $C \cup \{ \varnothing \}$ is a Cauchy cover of $P(M)$.
By definition, this implies $\top = \bigvee \{ \varepsilon_*(U) \mid U \in C \} \vee \varepsilon_*(\varnothing)$.
Since $M$ is proper, $\varepsilon_*(\varnothing) = \bot$, so we obtain $\top = \bigvee \{ \varepsilon_*(U) \mid U \in C \}$.
Thus, $C$ is also a Cauchy cover, completing the proof.
\end{proof}

The following simple lemma provides a useful characterization of the properness property for locales:

\begin{lem}[locale-proper]
A locale is proper if and only if the only open that has no points is the bottom element $\bot$.
\end{lem}
\begin{proof}
This result follows immediately from the definition of $\varepsilon_*(\varnothing)$.
\end{proof}

The following proposition demonstrates that the functor $P : \cat{Loc} \to \cat{Precov}$ restricts to a fully faithful embedding from the category of proper regular locales to the category of cover spaces:

\begin{prop}[locale-ff]
If $M$ is a proper locale and $N$ is a regular locale, then the function $P : \cat{Loc}(M,N) \to \cat{Cov}(P(M),P(N))$ is a bijection.
\end{prop}
\begin{proof}
We first prove injectivity.
Let $f,g : M \to N$ be locale maps such that $P(f) = P(g)$.
To prove $f = g$, it suffices to show $f^*(a) \leq g^*(a)$ for every $a \in O_N$, by symmetry.

Since $N$ is regular, we can write $a = \bigvee \{ b \in O_N \mid b \rb a \}$.
Thus, it is enough to show $f^*(b) \leq g^*(a)$ for every $b \rb a$.
From $\top = \neg b \vee a$, we have:
\[
f^*(b) = f^*(b) \wedge (g^*(\neg b) \vee g^*(a)).
\]
Therefore, it suffices to show that $f^*(b) \wedge g^*(\neg b) = \bot$.
By \rlem{locale-proper}, it is enough to show that the open $f^*(b) \wedge g^*(\neg b)$ has no points.
Suppose, for contradiction, that there is a completely prime filter $F$ containing $f^*(b) \wedge g^*(\neg b)$.
Then $F$ contains both $f^*(b)$ and $g^*(\neg b)$, which implies $P(f)(F) = P(g)(F)$ contains both $b$ and $\neg b$, which is a contradiction.
This proves injectivity.

Next, we prove surjectivity.
Let $f : P(M) \to P(N)$ be a cover map.
Define $g : M \to N$ as follows:
\[
g^*(a) = \bigvee \{ \varepsilon_*(f^{-1}(\varepsilon^*(b))) \mid b \rb a \}.
\]
It is straightforward to verify that $g^*$ preserves finite meets because $\varepsilon_*(f^{-1}(\varepsilon^*(-)))$ does.

To show that $g^*$ preserves joins, it suffices to prove: $g^*(\bigvee_{j \in J} a_j) \leq \bigvee_{j \in J} g^*(a_j)$.
This amounts to showing:
\[
\varepsilon_*(f^{-1}(\varepsilon^*(b))) \leq \bigvee \{ \varepsilon_*(f^{-1}(\varepsilon^*(b'))) \mid \exists j \in J, b' \rb a_j \},
\]
for every $b \in O_N$ such that $b \rb \bigvee_{j \in J} a_j$.

Since $N$ is regular, we have:
\[
\top = \neg b \vee \bigvee \{ b' \in O_N \mid \exists j \in J, b' \rb a_j \}.
\]
By \rlem{locale-cauchy}, the set
\[
\{ \varepsilon^*(\neg b) \} \cup \bigcup \{ \varepsilon^*(b') \mid \exists j \in J, b' \rb a_j \}
\]
is a Cauchy cover in $P(N)$.
Since $f$ is a cover map, the set
\[
\{ f^{-1}(\varepsilon^*(\neg b)) \} \cup \bigcup \{ f^{-1}(\varepsilon^*(b')) \mid \exists j \in J, b' \rb a_j \}
\]
is a Cauchy cover in $P(M)$.
Thus:
\[
\varepsilon_*(f^{-1}(\varepsilon^*(\neg b))) \vee \bigvee \{ \varepsilon_*(f^{-1}(\varepsilon^*(b'))) \mid \exists j \in J, b' \rb a_j \} = \top.
\]
It remains to show:
\[
\varepsilon_*(f^{-1}(\varepsilon^*(\neg b))) \wedge \varepsilon_*(f^{-1}(\varepsilon^*(b))) = \bot.
\]
This follows because $\varepsilon_*(f^{-1}(\varepsilon^*(-)))$ preserves meets and $M$ is proper:
$\varepsilon_*(f^{-1}(\varepsilon^*(\neg b))) \wedge \varepsilon_*(f^{-1}(\varepsilon^*(b))) = \varepsilon_*(f^{-1}(\varepsilon^*(\neg b \wedge b))) = \varepsilon_*(f^{-1}(\varepsilon^*(\bot))) = \varepsilon_*(f^{-1}(\varnothing)) = \varepsilon_*(\varnothing) = \bot$.
This completes the proof that $g^*$ preserves joins.

Finally, we show $P(g) = f$.
For any point $F \in P(M)$, we have:
\[
P(g)(F) = \{ a \in O_N \mid \exists b \rb a, \varepsilon_*(f^{-1}(\varepsilon^*(b))) \in F \}.
\]
First, we show $P(g)(F) \subseteq f(F)$.
Suppose $a \in P(g)(F)$.
Then there exists $b \rb a$ such that $\varepsilon_*(f^{-1}(\varepsilon^*(b))) \in F$.
Since $\varepsilon^*(\varepsilon_*(f^{-1}(\varepsilon^*(b)))) \subseteq f^{-1}(\varepsilon^*(b))$, it follows that $F \in f^{-1}(\varepsilon^*(b))$.
Thus, $a \in f(F)$.

Conversely, let $a \in f(F)$.
Since $N$ is regular, there exist opens $c,b \in f(F)$ such that $c \rb b$ and $b \rb a$.
By \rlem{locale-cauchy}, $\{ \varepsilon^*(\neg c), \varepsilon^*(b) \}$ is a Cauchy cover.
Since $f$ is a cover map, $\{ f^{-1}(\varepsilon^*(\neg c)), f^{-1}(\varepsilon^*(b)) \}$ is also a Cauchy cover.
Thus:
\[
\top = \varepsilon_*(f^{-1}(\varepsilon^*(\neg c))) \vee \varepsilon_*(f^{-1}(\varepsilon^*(b))).
\]
Since $F$ is a completely prime filter, either $\varepsilon_*(f^{-1}(\varepsilon^*(\neg c)))$ or $\varepsilon_*(f^{-1}(\varepsilon^*(b)))$ belongs to $F$.
The former is impossible because $\neg c \notin f(F)$, so $a \in P(g)(F)$.
This concludes the proof.
\end{proof}

Next, we construct a locale from a precover space.
We define the frame of opens of this locale using a method described in Section~2.11 of \cite{stone-spaces} and Section~0.1.3 of \cite{vickers-compact}, which we briefly recall here.

A \emph{coverage} on a meet-semilattice $P$ is a relation $\triangleleft_0$ between elements of $P$ and subsets of $P$, satisfying the following conditions:
\begin{itemize}
\item If $a \triangleleft_0 U$, then $b \leq a$ for every $b \in U$.
\item If $a \triangleleft_0 U$ and $b \leq a$, then there exists a set $V$ such that $b \triangleleft_0 V$ and, for every $c \in V$, it holds that $c \leq d$ for some $d \in U$.
\end{itemize}

Given a coverage $(P,\triangleleft_0)$, a subset $U$ of $P$ is called a \emph{$\triangleleft_0$-ideal} if it is downward closed and $a \in U$ whenever $a \triangleleft_0 V$ and $V \subseteq U$.
The set of $\triangleleft_0$-ideals forms a frame under the inclusion relation.
We denote this frame by $F(P,\triangleleft_0)$.
The idea behind this definition is that the frame $F(P, \triangleleft_0)$ is generated under joins by elements of $P$ under the relations specified by $\triangleleft_0$.

Given a coverage $(P,\triangleleft_0)$, it is convenient to define a new relation $\triangleleft$ as the closure of $\triangleleft_0$ under the following rules:
\begin{itemize}
\item If $a \in U$, then $a \triangleleft U$.
\item If $a \leq b$, then $a \triangleleft \{ b \}$.
\item If $a \triangleleft U$ and $b \triangleleft V$ for every $b \in U$, then $a \triangleleft V$.
\end{itemize}

Now, for any subset $U$ of $P$, we define $\overline{U} = \{ a \in P \mid a \triangleleft U \}$.
This set is the smallest $\triangleleft_0$-ideal containing $U$.
The join of a set $U$ is given by $\overline{\bigcup U}$.
Moreover, there exists a meet-preserving monotone map $[-] : P \to F(P,\triangleleft_0)$, defined by $[a] = \overline{ \{ a \} }$.
The frame $F(P,\triangleleft_0)$ is generated under joins by elements of the form $[a]$.
Additionally, we have the following equivalence:
\[
[a] \leq \bigvee_{j \in J} [b_j] \quad \text{if and only if} \quad a \triangleleft \{ b_j \mid j \in J \}.
\]

We are now ready to define a locale $L(X)$ for every precover space $X$.
The frame of opens of this locale is generated by subsets of $X$ with the following coverage:
\begin{itemize}
\item $U \triangleleft_0 \{ U \cap V \mid V \in C \}$ for every Cauchy cover $C$.
\item $U \triangleleft_0 \{ V \mid V \rb^s_X U \}$ for every set $U \subseteq X$.
\item $\varnothing \triangleleft_0 \varnothing$.
\end{itemize}

This locale is regular by construction, as shown in the following lemma:

\begin{lem}[cover-locale-regular]
For every precover space $X$, the locale $L(X)$ is regular.
\end{lem}
\begin{proof}
It suffices to verify regularity for the generating elements.
From the property $[U] \leq \bigvee \{ [V] \mid V \rb^s_X U \}$, it is enough to show that $V \rb^s_X U$ implies $[V] \rb [U]$.
To see this, note that $\{ X \backslash V, U \}$ is a Cauchy cover, which implies $\top = [X \backslash V] \vee [U]$.
Therefore, it remains to show that $[X \backslash V] \leq \neg [V]$.
This follows from the fact that 
\[
[X \backslash V] \wedge [V] = [(X \backslash V) \cap V] = [\varnothing] = \bot.
\]
\end{proof}

The following lemma provides a useful description of the frame of opens of $L(X)$ in terms of Cauchy covers of $X$:

\begin{lem}[locale-cover]
Let $X$ be a strongly regular cover space, and let $U$, $U'$, and $V_j$ be subsets of $X$ such that $U' \rb^s_X U$ and $[U] \leq \bigvee_{j \in J} [V_j]$.
Then $\{ X \backslash U' \} \cup \{ V_j \mid j \in J \}$ is a Cauchy cover of $X$.
\end{lem}
\begin{proof}
The proof proceeds by induction on the generation of $U \triangleleft \{ V_j \mid j \in J \}$.
First, let us consider the base case $\triangleleft_0$:
\begin{itemize}
\item If $\{ V_j \mid j \in J \} = \{ U \cap V \mid V \in C \}$ for a Cauchy cover $C$, then $\{ (X \backslash U') \cap V \mid V \in C \} \cup \{ U \cap V \mid V \in C \}$ is a Cauchy cover by \axref{CG}.
        This cover refines $\{ X \backslash U' \} \cup \{ U \cap V \mid V \in C \}$.
\item If $\{ V_j \mid j \in J \} = \{ V \mid V \rb^s_X U \}$, then $\{ X \backslash U' \} \cup \{ V \mid V \rb^s_X U \}$ is a Cauchy cover since $X$ is strongly regular.
\item If $U = \varnothing$ and $\{ V_j \mid j \in J \} = \{ V \mid V \rb^s_X U \} = \varnothing$, then $U' = \varnothing$, and $\{ X \backslash U' \}$ is a Cauchy cover by \axref{CT}.
\end{itemize}

The only non-trivial remaining case is the transitivity rule.
Suppose $U \triangleleft \{ W_i \mid i \in I \}$ and $W_i \triangleleft \{ V_j \mid j \in J \}$ for every $i \in I$.
By the induction hypothesis, $\{ X \backslash U'\} \cup \{ W_i \mid i \in I \}$ is a Cauchy cover.
Since $X$ is strongly regular, the set
\[ C = \{ X \backslash U'\} \cup \{ W \mid \exists i \in I, W \rb^s_X W_i \} \]
is also a Cauchy cover.
For each $W \in C$, we define $D_W$ as follows:
\[ D_W = \{ V \mid W = X \backslash U' \textrm{ or } \exists i \in I, W \rb^s_X W_i, V \in \{ X \backslash W \} \cup \{ V_j \mid j \in J \} \}. \]

This set is a Cauchy cover for every $W \in C$.
Indeed, if $W = X \backslash U'$, then $D_W$ is Cauchy as it refines $\{ X \}$.
Otherwise, there exists $i \in I$ such that $W \rb^s_X W_i$, and $D_W$ is Cauchy by the induction hypothesis.

By \axref{CG}, the set $\{ W \cap V \mid W \in C, V \in D_W \}$ is a Cauchy cover.
We show that it refines $\{ X \backslash U'\} \cup \{ V_j \mid j \in J \}$.
If $W = X \backslash U'$, then $W \cap V \subseteq X \backslash U'$ for any $V \in D_W$.
Otherwise, there exists $i \in I$ such that $W \rb^s_X W_i$.
In this case, either $V = X \backslash W$ or $V = V_j$ for some $j \in J$.
In the former case, $W \cap V = \varnothing \subseteq X \backslash U'$.
In the latter case, $W \cap V \subseteq V_j$.

This shows that $\{ X \backslash U'\} \cup \{ V_j \mid j \in J \}$ is a Cauchy cover, completing the proof.
\end{proof}

The following two lemmas follow easily from the previous one:

\begin{lem}[locale-top-cover]
If $X$ is a proper strongly regular cover space, then a set $C$ is a Cauchy cover of $X$ if and only if $\bigvee \{ [U] \mid U \in C \} = \top$.
\end{lem}
\begin{proof}
The ``only if'' direction is immediate from the definition of $L(X)$.
Conversely, suppose $\bigvee \{ [U] \mid U \in C \} = \top$.
Then, by \rlem{locale-cover}, the set $C \cup \{ \varnothing \}$ is a Cauchy cover.
Since $X$ is proper, $C$ itself must also be a Cauchy cover.
\end{proof}

\begin{lem}[locale-point-cover]
If $U$ is a neighborhood of a point $x$ in a strongly regular cover space and $[U] \leq \bigvee_{j \in J} [V_j]$, then $V_j$ is a neighborhood of $x$ for some $j \in J$.
\end{lem}
\begin{proof}
By \rlem{rbs-point}, we have $\{ x \} \rb^s U$.
From \rlem{locale-cover} and \axref{CR}, it follows that the set $\{ X \setminus \{ x \} \} \cup \{ V' \mid \exists j \in J, V' \rb V_j \}$ is a Cauchy cover.
Since $x \notin X \setminus \{ x \}$, it must belong to some $V'$ such that $V' \rb V_j$ for some $j \in J$.
Thus, $V_j$ is a neighborhood of $x$.
\end{proof}

Finally, we are ready to prove the main result of this section:

\begin{thm}[locale-equiv]
The functor $P : \cat{Loc} \to \cat{Precov}$ restricts to an equivalence between the category of proper regular locales and the category of proper strongly complete cover spaces.
\end{thm}
\begin{proof}
By \rprop{locale-ff}, \rprop{locale-regular}, and \rprop{locale-cover-proper}, the functor $P$ restricts to a fully faithful functor between the required categories.
We only need to show that it is essentially surjective on objects.

Let $X$ be a proper strongly regular cover space with.
Define a function $\eta_X : X \to P(L(X))$ by
\[ \eta_X(x) = \{ C \mid \exists U, \{ x \} \rb_X U, [U] \leq C \}. \]
It is clear that $\eta_X(x)$ is a filter.
To show that it is completely prime, we only need to consider joins of the form $\bigvee_{j \in J} [V_j]$ since the elements of the form $[U]$ generate $L(X)$.
Suppose $U$ is a neighborhood of $x$ such that $[U] \leq \bigvee_{j \in J} [V_j]$, and $\bigvee_{j \in J} [V_j]$ belongs to $\eta_X(x)$.
By \rlem{locale-point-cover}, $V_j$ is a neighborhood of $x$ for some $j \in J$.
Thus, $[V_j]$ belongs to $\eta_X(x)$, proving that $\eta_X(x)$ is a completely prime filter.

By \rlem{cover-locale-regular}, $L(X)$ is regular.
We now show that it is proper.
It suffices to prove that $[U] = \bot$ whenever $[U]$ has no points.
Since $[U] \leq \bigvee \{ [V] \mid V \rb^s_X U \}$, it is enough to show that $[V] \leq \bot$ whenever $V \rb^s_X U$.
If $x$ is a point of $V$, then $\eta_X(x)$ is a point of $[U]$, contradicting the assumption on $U$.
Hence, $V$ must be empty, implying $[V] \leq \bot$.
Thus, $L(X)$ is proper.

To show that $\eta_X$ is a weakly dense embedding, we first verify that it is a cover map.
Let $C$ be a Cauchy cover of $P(L(X))$.
By definition, 
\[
\top = \bigvee \{ \varepsilon_*(W) \mid W \in C \} = \bigvee \{ [U] \mid \exists W \in C, \varepsilon^*([U]) \subseteq W \}.
\]
By \rlem{locale-top-cover} and \axref{CR}, the set 
\[
\{ V \mid \exists U, V \rb U, \exists W \in C, \varepsilon^*([U]) \subseteq W \}
\]
is a Cauchy cover of $X$.
We now show that this cover refines $\{ \eta_X^{-1}(W) \mid W \in C \}$.
Let $V$, $U$, and $W$ be such that $V \rb U$, $W \in C$, and $\varepsilon^*([U]) \subseteq W$.
If $x \in V$, then $U$ is a neighborhood of $x$, so $[U]$ belongs to $\eta_X(x)$.
Consequently, $\eta_X(x) \in W$, implying $V \subseteq \eta_X^{-1}(W)$.
This proves that $\eta_X$ is a cover map.

Next, we show that $\eta_X$ is weakly dense.
Let $U$ be an open set in $P(L(X))$ that does not intersect the image of $\eta_X$.
We need to prove that $U$ is empty.
Suppose $U$ contains a point $F$.
Since $\{ F \} \rb^s_{P(L(X))} U$, we have 
\[
\bigvee \{ [V] \mid \varepsilon^*([V]) \subseteq U \} \vee \bigvee \{ b \mid F \notin \varepsilon^*(b) \} = \top.
\]
Since $F$ is completely prime and $F$ cannot contain an open $b$ such that $F \notin \varepsilon^*(b)$, it must contain an open $[V]$ such that $\varepsilon^*([V]) \subseteq U$.
Furthermore, because $[V] \leq \bigvee \{ [V'] \mid V' \rb^s_X V \}$, it also contains an open $[V']$ such that $V' \rb^s_X V$.
Since $F$ is completely prime, $V'$ cannot be empty.
Thus, we may assume that there exists a point $x \in V'$.
Since $\{ x \} \rb_X V$, we have $[V] \in \eta_X(x)$.
This implies $\eta_X(x) \in \varepsilon^*([V])$ and hence $\eta_X(x) \in U$, contradicting the assumption on $U$.
Thus, $\eta_X$ is weakly dense.

Finally, we prove that $\eta_X$ is an embedding.
We need to show that, for every Cauchy cover $C$ of $X$, the set $\{ V \mid \exists U \in C, \eta_X^{-1}(V) \subseteq U \}$ is a Cauchy cover of $P(L(X))$.
By \rlem{locale-top-cover}, this set is a Cauchy cover of $P(L(X))$ if and only if 
\[
\{ W \mid \exists U \in C, \exists V, \varepsilon^*([W]) \subseteq V, \eta_X^{-1}(V) \subseteq U \}
\]
is a Cauchy cover of $X$.
This holds because it is refined by $C$.
Thus, $\eta_X$ is an embedding.

If $X$ is strongly complete, then \rthm{weakly-dense-lift} implies that $\eta_X$ is an isomorphism, as it is a weakly dense embedding.
This concludes the proof.
\end{proof}

\section{Real numbers}
\label{sec:reals}

The aim of this section is to study the cover space of real numbers.
We will define this cover space, demonstrate that it is strongly complete, and show that it corresponds to the locale of real numbers through the equivalence described in the previous section.
Additionally, we will prove that the field operations on real numbers are cover maps.

The cover space of real numbers $\mathbb{R}$ is defined as the one induced by the usual Euclidean metric space structure on $\mathbb{R}$.
Since this cover space structure is derived from a metric space, its underlying topology is the usual topology on $\mathbb{R}$.
Using the identification from \rprop{dedekind-cuts-filters}, a set $U$ is a neighborhood of a Dedekind filter $F$ if and only if there exists an open interval $(a,b) \in F$ such that $(a,b) \subseteq U$.

There is an inclusion $\mathbb{Q} \to \mathbb{R}$ that maps each rational number to its neighborhood filter.
It is straightforward to verify that the cover space structure on $\mathbb{Q}$, induced by the metric space structure, is isomorphic to the transferred structure.

\begin{prop}[real-completion]
The cover space of real numbers is the completion of $\mathbb{Q}$.
\end{prop}
\begin{proof}
Since the underlying topology of $\mathbb{R}$ is the usual topology induced by the metric, $\mathbb{Q}$ is dense in $\mathbb{R}$.
The inclusion $\mathbb{Q} \to \mathbb{R}$ is also an embedding by definition.
Moreover, because the underlying cover space of any metric space is always separated, \rlem{complete-part} and \rprop{dedekind-cauchy} imply that $\mathbb{R}$ is complete.
\end{proof}

In Section~\ref{sec:strongly-regular}, we defined a stronger version of completeness.
Interestingly, for certain spaces, this stronger version coincides with the weaker one.
To demonstrate this, we first introduce a few definitions:

\begin{defn}
Let $\mathcal{B}$ be a base of a cover space $X$, and let $U,V$ be subsets of $X$.
We say that $V$ is \emph{strongly $\mathcal{B}$-rather below} $U$, written $V \rb^s_\mathcal{B} U$, if $\{ X \backslash V, U \} \in \mathcal{B}$.

A base $\mathcal{B}$ is called \emph{strongly regular} if, for every $C \in \mathcal{B}$,
the collection $\{ V \mid \exists U \in C, V \rb^s_\mathcal{B} U \}$ is a Cauchy cover.
\end{defn}

\begin{remark}
If a cover space $X$ has a strongly regular base, then $X$ is strongly regular.
Conversely, the collection of all Cauchy covers of a strongly regular cover space forms a strongly regular base.
\end{remark}

\begin{example}
The collection of uniform covers of a metric space is a strongly regular base.
\end{example}

\begin{defn}
Let $\mathcal{B}$ be a base for a cover space $X$.
We say that $\mathcal{B}$ is \emph{strongly proper} if, for every cover $C \in \mathcal{B}$, the subset of $C$ consisting of inhabited sets belongs to $\mathcal{B}$.
\end{defn}

\begin{remark}
If the collection of all Cauchy covers of a cover space $X$ forms a strongly proper base, then $X$ is proper.
Classically, the converse is also true, but constructively, this is not always the case.
\end{remark}

\begin{example}
The collection of uniform covers of a metric space is a strongly proper base.
\end{example}

The following lemma demonstrates that strongly regular weakly Cauchy covers in a cover space $X$ satisfy a stronger regularity condition if $X$ has a strongly regular base:

\begin{lem}[cauchy-filter-base]
Let $X$ be a cover space with a strongly regular base $\mathcal{B}$, and let $F$ be a strongly regular weakly Cauchy cover in $X$.
Then, for every $U \in F$, there exists $V \in F$ such that $V \rb^s_\mathcal{B} U$.
\end{lem}
\begin{proof}
Since $F$ is strongly regular, there exists $V \in F$ such that $V \rb^s U$.
Because $\mathcal{B}$ is strongly regular, \rlem{subbase-regular} implies that the collection $\{ X \backslash V \} \cup \{ W \mid W \rb^s_\mathcal{B} U \}$ is a Cauchy cover.
As $F$ is a weakly Cauchy filter, there exists $W \in F$ such that $W \rb^s_\mathcal{B} U$.
\end{proof}

We can show that strongly regular weakly Cauchy filters are Cauchy under appropriate conditions:

\begin{prop}
If a cover space $X$ has a strongly proper strongly regular base $\mathcal{B}$, then every strongly regular weakly Cauchy filter is Cauchy.
\end{prop}
\begin{proof}
Let $F$ be a strongly regular weakly Cauchy filter, and suppose $U \in F$.
By \rlem{cauchy-filter-base}, there exists $V \in F$ such that $V \rb^s_\mathcal{B} U$.
Since $\mathcal{B}$ is strongly proper, the collection $\{ X \backslash V \} \cup \{ U \mid \exists x \in U \}$ is a Cauchy cover.
As $F$ satisfies the Cauchy condition, either $U$ is inhabited or $F$ contains the complement of $V$.
However, since $F$ is weakly proper, the latter case is impossible.
Therefore, $U$ must be inhabited, completing the proof.
\end{proof}

\begin{cor}
If a complete cover space has a strongly proper strongly regular base, then it is strongly complete.
\end{cor}

\begin{cor}[metric-complete]
A metric space is complete if and only if it is strongly complete.
\end{cor}

Next, we show that the localic reals correspond to the cover space of real numbers.
Recall that the frame of opens of the locale $\mathbb{R}_l$ of real numbers is generated by the set of open intervals with rational endpoints, along with the empty set, under the following coverage relation:
\begin{itemize}
\item $(a,d) \triangleleft_0 \{ (a,c), (b,d) \}$ for all rational numbers $a < b < c < d$.
\item $(a,d) \triangleleft_0 \{ (b,c) \mid a < b < c < d \}$ for all rational numbers $a < d$.
\item $\varnothing \triangleleft_0 \varnothing$.
\end{itemize}

While the locale $L(\mathbb{R})$ is generated by all subsets of $\mathbb{R}$, the locale $\mathbb{R}_l$ is generated by a much smaller set of open intervals with rational endpoints.
To establish an isomorphism between these locales, it is convenient to identify a smaller generating set for $L(\mathbb{R})$.
To achieve this, we introduce the following definition:

\begin{defn}
Let $\mathcal{G}$ be a set of subsets of a strongly regular cover space $X$.
We say that $\mathcal{G}$ \emph{generates} $X$ if there exists a strongly regular base $\mathcal{B}$ such that every cover in $\mathcal{B}$ is refined by a Cauchy cover consisting entirely of elements from $\mathcal{G}$.
\end{defn}

\begin{example}
Any metric space is generated by the set of open balls.
\end{example}

The following lemma shows that stronger regularity conditions on a cover space imply stronger regularity properties for the corresponding locale:

\begin{lem}[gen-regular]
Let $\mathcal{B}$ be a strongly regular base of a cover space $X$.
Then, for every $U \subseteq X$, we have $[U] = \bigvee \{ [V] \mid V \rb^s_\mathcal{B} U \}$ in $L(X)$.
\end{lem}
\begin{proof}
By the coverage relation of $L(X)$, we know that $[U] = \bigvee \{ [V] \mid V \rb^s U \}$.
Thus, it suffices to show that $[V] \leq \bigvee \{ [W] \mid W \rb^s_\mathcal{B} U \}$ for every $V \rb^s U$.
Since $\mathcal{B}$ is strongly regular, the set $\{ X \backslash V \} \cup \{ W \mid W \rb^s_\mathcal{B} U \}$ is a Cauchy cover.
The coverage relation of $L(X)$ then implies the desired equation.
\end{proof}

The following lemma demonstrates how the concept of generation translates into a property of the locale associated with $X$, providing further motivation for the terminology.

\begin{lem}[locale-gen]
Let $X$ be a precover space generated by a set $\mathcal{G}$.
Then $L(X)$ is generated under joins by elements of the form $[G]$ for $G \in \mathcal{G}$.
\end{lem}
\begin{proof}
Since $L(X)$ is generated by elements of the form $[U]$ for $U \subseteq X$, it suffices to show that every such element can be expressed as a join of elements of the form $[G]$ for $G \in \mathcal{G}$.
By \rlem{gen-regular}, we know that $[U] = \bigvee \{ [V] \mid V \rb^s_\mathcal{B} U \}$ for any $U \subseteq X$.
If $V \rb^s_\mathcal{B} U$, then by the defining property of $\mathcal{B}$, the collection $\{ X \backslash V \} \cup \{ G \in \mathcal{G} \mid G \subseteq U \}$ is a Cauchy cover.
Consequently, $[V] \leq \bigvee \{ [G] \mid G \in \mathcal{G}, G \subseteq U \}$.
Thus, $[U] = \bigvee \{ [G] \mid G \in \mathcal{G}, G \subseteq U \}$, completing the proof.
\end{proof}

\begin{prop}
The locale of real numbers $\mathbb{R}_l$ is isomorphic to $L(\mathbb{R})$.
\end{prop}
\begin{proof}
Define a map $f : L(\mathbb{R}) \to \mathbb{R}_l$ by:
\[ f^*((a,b)) = [\{ x \in \mathbb{R} \mid a < x < b \}], \quad f^*(\varnothing) = [\varnothing]. \]
It is straightforward to verify that $f^*$ respects the coverage relation on $\mathbb{R}_l$.
To prove that $f$ is an isomorphism, it suffices to show that $f^*$ is bijective.

The map $f^*$ is clearly injective by construction, as distinct intervals map to distinct sets.
By \rlem{locale-gen}, the locale $L(\mathbb{R})$ is generated by elements of the form $[B_\varepsilon(x)]$, which implies that $f^*$ is surjective.
\end{proof}

\begin{cor}
The cover space of real numbers is isomorphic to $P(\mathbb{R}_l)$.
\end{cor}
\begin{proof}
By \rprop{real-completion} and \rcor{metric-complete}, the cover space $\mathbb{R}$ is strongly complete.
By \rthm{locale-equiv} and the previous proposition, we have the isomorphism $\mathbb{R} \simeq P(L(\mathbb{R})) \simeq P(\mathbb{R}_l)$.
\end{proof}

This corollary establishes a bijection between localic maps $M \to \mathbb{R}_l$ and cover maps $P(M) \to \mathbb{R}$ for proper locales $M$.
Consequently, it provides a convenient framework for constructing localic real-valued maps.
In general, the condition for a function $f : X \to \mathbb{R}$ to be a cover map can appear complicated.
However, there often exists a simple sufficient condition that ensures a function is a cover map.

\begin{defn}
Let $X$ and $Y$ be cover spaces, and let $\mathcal{B}_X$ be a subbase for $X$.
A function $f : X \to Y$ is said to be \emph{locally uniform} (relative to $\mathcal{B}_X$) if, for every Cauchy cover $E$ of $Y$,
there exists $C \in \mathcal{B}_X$ such that, for every $U \in C$, we have
\[ \{ V \mid \exists W \in E, U \cap V \subseteq f^{-1}(W) \} \in \mathcal{B}_X. \]
\end{defn}

If $X$ is a uniform space, we will always consider locally uniform functions relative to the subbase of uniform covers.
Unfolding the definition of uniform covers in metric spaces, we obtain the following proposition:

\begin{prop}
Let $X$ and $Y$ be metric spaces.
A function $f : X \to Y$ is locally uniform if and only if, for every $\varepsilon > 0$, there exists $\delta_1 > 0$ such that,
for every $x_1 \in X$, there exists $\delta_2 > 0$ such that, for every $x_2 \in X$, there exists $y \in Y$ satisfying $B_{\delta_1}(x_1) \cap B_{\delta_2}(x_2) \subseteq f^{-1}(B_\varepsilon(y))$.
\end{prop}

The next proposition establishes that every locally uniform function is a cover map.
In the following section, we will see that for certain cover spaces, the converse is also true.

\begin{prop}[locally-uniform-cover-map]
Every locally uniform function is a cover map.
\end{prop}
\begin{proof}
Let $f : X \to Y$ be a locally uniform function.
It suffices to show that for every Cauchy cover $E$ of $Y$, the cover $\{ f^{-1}(W) \mid W \in E \}$ is Cauchy.
Let $E$ be a Cauchy cover of $Y$.
By the definition of local uniformity, there exists a Cauchy cover $C \in \mathcal{B}_X$ such that, for every $U \in C$, the set
$D_U = \{ V \mid \exists W \in E, U \cap V \subseteq f^{-1}(W) \}$
belongs to $\mathcal{B}_X$.
It follows that the cover $\{ U \cap V \mid U \in C, V \in D_U \}$ is Cauchy.
Since it refines $\{ f^{-1}(W) \mid W \in E \}$, the proof is completed.
\end{proof}

Finally, we will demonstrate that the field operations on $\mathbb{R}$ are cover maps.
First, we establish that dense embeddings are closed under products:

\begin{lem}
Let $X$ and $X'$ be precover spaces, $Y$ and $Y'$ be cover spaces, and $f : X \to Y$ and $g : X' \to Y'$ be cover maps.
If $f$ and $g$ are dense embeddings, then so is $f \times g : X \times X' \to Y \times Y'$.
\end{lem}
\begin{proof}
Since embeddings are closed under products by \rlem{prod-embedding}, it suffices to show that $f \times g$ is dense whenever $f$ and $g$ are dense.

Let $(y,y') \in Y \times Y'$ and let $U$ be a neighborhood of $(y,y')$.
By \rlem{rb-point}, there exists a neighborhood $V$ of $(y,y')$ such that $V \rb U$.
Applying \rlem{cover-map-rb} to the map $\langle \mathrm{id}_Y, \mathrm{const}(y') \rangle$, we deduce that the set $\{ y_1 \mid (y_1,y') \in V \}$ is a neighborhood of $y$.
Since $f$ is dense, there exists $x \in X$ such that $(f(x),y') \in V$.
It follows that $U$ is a neighborhood of $(f(x),y')$.

Now, applying \rlem{cover-map-rb} to the map $\langle \mathrm{const}(f(x)), \mathrm{id}_{Y'} \rangle$ and the neighborhood $U$,
we find that the set $\{ y_2 \mid (f(x),y_2) \in U \}$ is a neighborhood of $(f(x),y')$.
Since $g$ is dense, there exists $x' \in X'$ such that $(f(x),f(x')) \in U$.
Thus, $f \times g$ is dense.
\end{proof}

This lemma implies that we can extend cover maps $\mathbb{Q}^n \to \mathbb{R}$ to cover maps $\mathbb{R}^n \to \mathbb{R}$.
Since the ring operations on $\mathbb{Q}$ are locally uniform, \rprop{locally-uniform-cover-map} ensures that they extend to cover maps on $\mathbb{R}$.
By the uniqueness of extensions, these maps satisfy the ring axioms.

We now aim to show that the inverse function is a cover map.
Let $\mathbb{R}_*$ denote the union of negative and positive real numbers.
It is well-known that a real number is invertible if and only if it belongs to $\mathbb{R}_*$.
Our goal is to prove that the function $(-)^{-1} : \mathbb{R}_* \to \mathbb{R}$ is a cover map.

If we equip $\mathbb{R}_*$ with the induced cover space structure from $\mathbb{R}$, this map is not a cover map.
The issue lies in the fact that $\mathbb{R}_*$ is incomplete under this cover space structure, and its completion is the entire real line $\mathbb{R}$.
If such a cover map could be defined, it would necessarily extend to $\mathbb{R}$, which is impossible.
Therefore, we must impose a different cover space structure on $\mathbb{R}_*$.

To address this, we introduce a more general construction that applies to any open subset of a complete cover space.
Let $S$ be an open subset of a cover space $X$.
We define the following subbase on $S$:
\[ \mathcal{B}_S = \{ \{ V \cap S \mid V \in D \} \mid D \in \mathcal{C}_X \} \cup \{ \{ V \mid V \rb_X V_1 \rb_X \ldots \rb_X V_n \rb_X S \} \mid n \in \mathbb{N} \}. \]

\begin{lem}
Let $S$ be an open subset of a cover space $X$.
Then $(S,\overline{\mathcal{B}_S})$ is a cover space.
\end{lem}
\begin{proof}
First, observe that $\mathcal{B}_S$ consists of covers.
For sets of the form $\{ V \cap S \mid V \in D \}$, this follows from the fact that $D$ is a cover of $X$.
For sets of the form $\{ V \mid V = V_1 \rb_X \ldots \rb_X V_n = S \}$, this follows from \rlem{rb-point} and the openness of $S$.
The regularity of $(S,\overline{\mathcal{B}_S})$ follows from \rlem{subbase-regular} since covers in $\mathcal{B}_S$ satisfy the conditions of this lemma by construction.
\end{proof}

The following lemma establishes that neighborhoods in $(S,\overline{\mathcal{B}_X})$ coincide with those in $X$:

\begin{lem}[subspace-neighborhood]
Let $S$ be an open subset of a cover space $X$.
If $x$ is a point in $S$ and $U$ is a subset of $S$, then $U$ is a neighborhood of $x$ in $S$ if and only if it is a neighborhood of $x$ in $X$.
\end{lem}
\begin{proof}
Since the inclusion function $S \to X$ is a cover map, \rlem{cover-map-rb} implies the ``if'' direction.
For the converse, observe that $F_x = \{ W \subseteq S \mid \{ x \} \rb_X W \}$ is a filter on $S$.
It is straightforward to verify that this filter intersects every element of $\mathcal{B}_X$ using \rlem{rb-point} and the fact that $S$ is open.
By \rlem{closure-filter}, $F_x$ also intersects every element of $\overline{\mathcal{B}_X}$.
If $U$ is a neighborhood of $x$ in $S$, then $\{ W \mid x \in W \implies W \subseteq U \}$ belongs to $\overline{\mathcal{B}_X}$.
Consequently, there exists a set $W$ such that $\{ x \} \rb_X W$ and $W \subseteq U$.
Thus, $U$ is a neighborhood of $x$ in $X$.
\end{proof}

We can now show that $(S,\overline{\mathcal{B}_S})$ is complete whenever $X$ is complete:

\begin{prop}
Let $S$ be an open subset of a complete cover space $X$.
Then $(S,\overline{\mathcal{B}_S})$ is a complete cover space.
\end{prop}
\begin{proof}
Since the inclusion function $S \to X$ is a cover map, \rprop{embedding-injective} implies that $S$ is separated.
We now show that $S$ is complete.
Suppose $F$ is a Cauchy filter on $S$.
Define $G = \{ V \mid V \cap S \in F \}$, which is a Cauchy filter on $X$.
By \rlem{subspace-neighborhood}, it suffices to show that $G^\vee$ belongs to $S$.

Since $\{ V \mid V \rb_X S \}$ is a Cauchy cover of $S$, there exists $V \in F$ such that $V \rb_X S$.
Now, \rlem{filter-point-char} implies that $S$ is a neighborhood of $G^\vee$.
This completes the proof.
\end{proof}

We now return to the problem of proving that the inverse function $(-)^{-1} : \mathbb{R}_* \to \mathbb{R}_*$ is a cover map.
For this, we assume that $\mathbb{R}_*$ is equipped with the cover space structure as described above.
To proceed, we provide a simpler characterization of this structure on $\mathbb{R}_*$:

\begin{lem}
Let $V$ be a subset of $\mathbb{R}$.
The following conditions are equivalent:
\begin{enumerate}
\item \label{it:inv-srb} $V \rb^s_\mathbb{R} \mathbb{R}_*$.
\item \label{it:inv-rb} $V \rb_\mathbb{R} \mathbb{R}_*$.
\item \label{it:inv-char} There exists a rational $\varepsilon > 0$ such that $V \cap B_\varepsilon(0) = \varnothing$.
\end{enumerate}
\end{lem}
\begin{proof}
It is clear that \eqref{it:inv-srb} implies \eqref{it:inv-rb}.
We now show that \eqref{it:inv-rb} implies \eqref{it:inv-char}. 
Suppose $V \rb_\mathbb{R} \mathbb{R}_*$.
Then the family $\{ W' \mid \exists W, W' \rb_\mathbb{R} W, \overlap{W}{V} \implies W \subseteq \mathbb{R}_* \}$ is a cover of $\mathbb{R}$.
Consequently, there exists a neighborhood $W$ of $0$ such that $\overlap{W}{V} \implies W \subseteq \mathbb{R}_*$.
Since $0 \notin \mathbb{R}_*$, this neighborhood $W$ does not intersect $V$.
Therefore, there exists an open ball around $0$ that does not intersect $V$.

Finally, we show that \eqref{it:inv-char} implies \eqref{it:inv-srb}.
To do so, it suffices to show that the Cauchy cover $\{ B_{\varepsilon/4}(x) \mid x \in \mathbb{R} \}$ refines $\{ \mathbb{R} \setminus V, \mathbb{R}_* \}$.
Let $x \in \mathbb{R}$.
By \axref{DS}, one of the following holds: $-\varepsilon/2 < x < \varepsilon/2$, $x < -\varepsilon/4$, or $x > \varepsilon/4$.
If $-\varepsilon/2 < x < \varepsilon/2$, then $B_{\varepsilon/4}(x) \subseteq \mathbb{R} \setminus V$.
If $x < -\varepsilon/4$ or $x > \varepsilon/4$, then $B_{\varepsilon/4}(x) \subseteq \mathbb{R}_*$.
This completes the proof.
\end{proof}

\begin{prop}
The cover space structure on $\mathbb{R}_*$ is generated by the following subbase:
\[ \mathcal{B}_{\mathbb{R}_*} = \{ \{ B_\varepsilon(x) \cap \mathbb{R}_* \mid x \in \mathbb{Q}, x \neq 0 \} \mid \varepsilon > 0 \} \cup \{ \{ (- \infty, - \varepsilon) \cup (\varepsilon, \infty) \mid \varepsilon > 0 \} \}. \]
\end{prop}
\begin{proof}
The first set in this union generates the transferred cover space structure on $\mathbb{R}_*$, while the second set generates the same covers as
\[ \{ \{ V \mid V \rb_\mathbb{R} V_1 \rb_\mathbb{R} \ldots \rb_\mathbb{R} V_n \rb_\mathbb{R} \mathbb{R}_* \} \mid n \in \mathbb{N} \} \]
as shown by the previous lemma.
Therefore, this subbase generates the same set of covers as the one described earlier for a general open set.
\end{proof}

We are now ready to prove that the inverse function is a cover map:

\begin{prop}
The inverse function $(-)^{-1} : \mathbb{R}_* \to \mathbb{R}_*$ is a cover map.
\end{prop}
\begin{proof}
Let $e : \mathbb{R}_* \to \mathbb{R}_*$ denote the inverse function.
First, observe that the cover $\{ e^{-1}((- \infty, - \varepsilon) \cup (\varepsilon, \infty)) \mid \varepsilon > 0 \}$ is refined by the uniform cover consisting of open balls of radius $1$.
Thus, it suffices to show that $\{ e^{-1}(B_\varepsilon(y) \cap \mathbb{R}_*) \mid y \in \mathbb{Q}, y \neq 0 \}$ is a Cauchy cover for every $\varepsilon > 0$.

We will prove that this cover is refined by the Cauchy cover
\[ \left\{ ((- \infty, - \delta) \cup (\delta, \infty)) \cap B_{\frac{\varepsilon \delta^2}{1 + \varepsilon \delta}}(x) \mid \delta > 0, x \in \mathbb{Q}, x \neq 0 \right\}. \]
To establish this, we need to show that
\[ ((- \infty, - \delta) \cup (\delta, \infty)) \cap B_{\frac{\varepsilon \delta^2}{1 + \varepsilon \delta}}(x) \subseteq e^{-1}\left(B_\varepsilon\left(\frac{1}{x}\right) \cap \mathbb{R}_*\right). \]

Let $z$ be a real number such that $\abs{z} > \delta$ and $\abs{z - x} < \frac{\varepsilon \delta^2}{1 + \varepsilon \delta}$.
We need to show that $\abs{\frac{1}{z} - \frac{1}{x}} < \varepsilon$.
First, note that $\abs{x} \geq \abs{z} - \abs{z - x} > \delta - \frac{\varepsilon \delta^2}{1 + \varepsilon \delta} = \frac{\delta}{1 + \varepsilon \delta}$.
Finally, we can complete the proof: $\abs{\frac{1}{z} - \frac{1}{x}} = \frac{\abs{z - x}}{\abs{z} \abs{x}} < \varepsilon$.
\end{proof}

\section{Compactness}
\label{sec:compact}

In this section, we define compact and totally bounded spaces and use these notions to characterize cover maps for certain cover spaces.

A cover space $X$ is said to be \emph{totally bounded} if, for every Cauchy cover $C$, there exists a finite subcover of $C$.
Explicitly, this means there exist sets $U_1, \ldots, U_n \in C$ such that $X = U_1 \cup \ldots \cup U_n$.
A cover space is called \emph{compact} if it is both complete and totally bounded.

Note that we do not require the finite subcover $\{ U_1, \ldots, U_n \}$ to be Cauchy.
However, under mild additional assumptions, we can choose a finite Cauchy subcover:

\begin{prop}[cauchy-fin]
Let $X$ be a cover space with a strongly proper regular base $\mathcal{B}$.
Then $X$ is totally bounded if and only if every Cauchy cover of $X$ has a finite subcover that belongs to $\mathcal{B}$.
\end{prop}
\begin{proof}
The ``if'' direction is obvious.
To prove the converse, suppose $C$ is a Cauchy cover of $X$.
Since $\mathcal{B}$ is regular, the set $\{ V \mid \exists U \in C, V \rb_\mathcal{B} U \}$ is also a Cauchy cover.
As $X$ is totally bounded, there exist sets $V_1, \ldots, V_n$ and corresponding sets $U_1, \ldots, U_n \in C$ such that $\{ V_1, \ldots, V_n \}$ is a cover of $X$ and $V_i \rb_\mathcal{B} U_i$ for all $i$.
Now, by \axref{BI} and the fact that $\mathcal{B}$ is strongly proper, the set
\[ E = \{ W_1 \cap \ldots \cap W_n \mid (\forall i, \overlap{W_i}{V_i} \implies W_i \subseteq U_i), \exists x \in \bigcap_i W_i \} \]
belongs to $\mathcal{B}$.

To complete the proof, we show that $E$ refines $\{ U_1, \ldots, U_n \}$.
Let $W_1, \ldots, W_n$ be sets satisfying $\overlap{W_i}{V_i} \implies W_i \subseteq U_i$ for all $i$, and suppose there exists a point $x \in \bigcap_i W_i$.
Since $\{ V_1, \ldots, V_n \}$ is a cover of $X$, there exists some $i$ such that $x \in V_i$.
It follows that $W_1 \cap \ldots \cap W_n \subseteq W_i \subseteq U_i$, which proves that $E$ refines $\{ U_1, \ldots, U_n \}$.
\end{proof}

For strongly regular bases, the conditions of the previous proposition can be relaxed.
To achieve this, we introduce the following definition:

\begin{defn}
Let $\mathcal{B}$ be a base for a cover space $X$.
We say that $\mathcal{B}$ is \emph{proper} if $C \cup \{ \varnothing \} \in \mathcal{B}$ implies $C \in \mathcal{B}$ for every cover $C$.
\end{defn}

\begin{example}
A cover space $X$ is proper if and only if the collection of all Cauchy covers of $X$ is a proper base.
\end{example}

\begin{example}
Any base of an inhabited cover space is proper.
\end{example}

Now, we are ready to prove a version of \rprop{cauchy-fin} for strongly regular bases:

\begin{prop}[cauchy-ws-fin]
Let $X$ be a cover space with a proper strongly regular base $\mathcal{B}$.
Then $X$ is totally bounded if and only if every Cauchy cover of $X$ has a finite subcover that belongs to $\mathcal{B}$.
\end{prop}
\begin{proof}
The ``if'' direction is obvious.
To prove the converse, suppose $C$ is a Cauchy cover of $X$.
Since $\mathcal{B}$ is strongly regular, the set $\{ V \mid \exists U \in C, V \rb^s_\mathcal{B} U \}$ is also a Cauchy cover.
As $X$ is totally bounded, there exist sets $V_1, \ldots, V_n$ and corresponding sets $U_1, \ldots, U_n \in C$ such that $\{ V_1, \ldots, V_n \}$ forms a cover of $X$ and $V_i \rb^s_\mathcal{B} U_i$ for all $i$.
By \axref{BI}, the set
\[ E = \{ W_1 \cap \ldots \cap W_n \mid \forall i, W_i \in \{ X \backslash V_i, U_i \} \} \]
belongs to $\mathcal{B}$.

Since $\mathcal{B}$ is proper, it suffices to show that $E$ refines $\{ U_1, \ldots, U_n, \varnothing \}$.
Consider an element $W_1 \cap \ldots \cap W_n \in E$.
Either $W_i = X \backslash V_i$ for every $i$ or there exists $i$ such that $W_i = U_i$.
In the former case, $W_i \cap \ldots \cap W_n = \varnothing$ since $\{ V_1, \ldots, V_n \}$ is a cover.
In the latter case, $W_i \cap \ldots \cap W_n \subseteq W_i = U_i$.
Thus, $E$ refines $\{ U_1, \ldots, U_n, \varnothing \}$, completing the proof.
\end{proof}

The previous two propositions imply that Cauchy covers often take a simpler form in totally bounded cover spaces:

\begin{cor}[tb-cauchy-uni]
Let $X$ be a totally bounded cover space, and let $\mathcal{B}$ be a base for $X$ which is either strongly proper and regular or proper and strongly regular.
Then the collection of Cauchy covers of $X$ coincides with $\mathcal{B}$.
\end{cor}

\begin{cor}
In a totally bounded metric space, Cauchy covers coincide with uniform covers.
\end{cor}

The following proposition provides a useful characterization of totally bounded cover spaces:

\begin{prop}[tb-aux]
Let $\mathcal{B}_X$ be a subbase for a cover space $X$.
Then $X$ is totally bounded if and only if every Cauchy cover in $\mathcal{B}_X$ has a finite subcover.
\end{prop}
\begin{proof}
The proof follows directly by induction on the construction of $\overline{\mathcal{B}}$.
\end{proof}

This proposition yields the usual definition of total boundedness in the cases of uniform spaces and metric spaces:

\begin{cor}
A uniform space is totally bounded if and only if every uniform cover has a finite subcover.
\end{cor}

\begin{cor}[tb-metric]
A metric space $X$ is totally bounded if and only if, for every $\varepsilon > 0$,
there exist points $x_1, \ldots, x_n \in X$ such that every point of $X$ is $\varepsilon$-close to one of the points $x_1, \ldots, x_n$.
\end{cor}

\begin{example}
Every bounded subset of $\mathbb{R}^n$ satisfies the conditions of \rcor{tb-metric}.
Therefore, a subset of $\mathbb{R}^n$ is totally bounded if and only if it is bounded.
\end{example}

We can derive the Tychonoff theorem as an immediate consequence of the previous propositions.
While the theorem is not constructively true for topological spaces, it does hold for locales, albeit with a non-trivial proof (see \cite{vickers-tychonoff,vickers-compact}).
In the context of cover spaces, the result becomes straightforward:

\begin{prop}
Let $\{ X_i \}_{i \in I}$ be a collection of totally bounded cover spaces.
Then the product $\prod_{i \in I} X_i$ is also totally bounded.
\end{prop}
\begin{proof}
The cover space structure on $\prod_{i \in I} X_i$ is generated by covers of the form $C' = \{ \pi_i^{-1}(U) \mid U \in C \}$, where  $i \in I$ and $C$ is a Cauchy cover of $X_i$.
By \rprop{tb-aux}, it suffices to show that $C'$ has a finite subcover.
Since $X_i$ is totally bounded, there exist sets $U_1, \ldots, U_n \in C$ that cover $X_i$.
Then $\pi_i^{-1}(U_1), \ldots, \pi_i^{-1}(U_n)$ form a finite subcover of $C'$, completing the proof.
\end{proof}

\begin{cor}
Let $\{ X_i \}_{i \in I}$ be a collection of compact cover spaces.
Then the product $\prod_{i \in I} X_i$ is also compact.
\end{cor}
\begin{proof}
The subcategory of complete cover spaces is reflective in the category of cover spaces, which implies that it is closed under products.
Therefore, the result follows directly from the previous proposition.
\end{proof}

The following proposition provides a useful characterization of Cauchy covers under appropriate assumptions:

\begin{prop}
Let $X$ be a cover space with a strongly regular base $\mathcal{B}$, and let $C$ be a Cauchy cover of $X$ consisting of totally bounded inhabited sets.
Then a set $E$ is a Cauchy cover of $X$ if and only if $\{ V \mid \exists W \in E, U \cap V \subseteq W \} \in \mathcal{B}$, for every $U \in C$.
\end{prop}
\begin{proof}
If $\{ V \mid \exists W \in E, U \cap V \subseteq W \} \in \mathcal{B}$ for every $U \in C$, then \axref{CG} implies that $E$ is a Cauchy cover.
Conversely, suppose that $E$ is a Cauchy cover, and let $U \in C$.
The strongly regular base $\mathcal{B}$ can be transferred to $U$ in the same way that cover space structures are transferred.
Since $U$ is inhabited, the transferred base $\mathcal{B}_U$ is proper.
By \rcor{tb-cauchy-uni}, we have $\{ U \cap W \mid W \in E \} \in \mathcal{B}_U$.
This means that $\{ V \mid \exists W \in E, U \cap V \subseteq U \cap W \} \in \mathcal{B}$, which establishes the desired property.
\end{proof}

\begin{example}
Let $X$ be a metric space in which bounded sets are totally bounded, and let $\delta_1 > 0$.
Define $C$ as the set of open balls of radius $\delta_1$.
Then a set $E$ is a Cauchy cover of $X$ if and only if, for every $x_1 \in X$, there exists $\delta_2 > 0$ such that, for every $x_2 \in X$,
there exists $W \in E$ such that $B_{\delta_1}(x_1) \cap B_{\delta_2}(x_2) \subseteq W$.
\end{example}

\begin{cor}
Let $X$ be a cover space with a strongly regular base $\mathcal{B}$ such that the set of totally bounded inhabited sets belongs to $\mathcal{B}$.
Then a function $X \to Y$ is a cover map if and only if it is locally uniform with respect to $\mathcal{B}$.
\end{cor}

\begin{cor}
Let $X$ be a metric space in which every bounded inhabited set is totally bounded.
Then a function $X \to Y$ is a cover map if and only if it is locally uniform.
\end{cor}

\begin{cor}
A function $\mathbb{R}^n \to Y$ is a cover map if and only if it is locally uniform.
\end{cor}

We conclude this section with a characterization of compact subspaces.
Recall that a subspace of a cover space is closed if it contains all of its limit points.
The following proposition provides a characterization of complete subspaces within a complete cover space:

\begin{prop}
A subspace of a complete cover space is complete if and only if it is closed.
\end{prop}
\begin{proof}
First, suppose that $S$ is a closed subspace of a complete space $X$.
By \rprop{embedding-injective}, $S$ is separated.
Let us prove that it is complete.
Let $F$ be a Cauchy filter on $S$.
Then $G = \{ U \mid U \cap S \in F \}$ is a Cauchy filter on $X$.

To show that $G^\vee$ belongs to $S$, it suffices to prove that $G^\vee$ is a limit point of $S$ since $S$ is closed.
Let $U$ be a neighborhood of $G^\vee$.
By \rlem{filter-point-char}, there exists $V \in G$ such that $V \rb_X U$.
Since $F$ is proper and $V \cap S \in F$, there exists a point $x \in V \cap S$.
Thus, $U$ and $S$ intersect, which implies that $G^\vee$ is a limit point of $S$.

Next, we need to prove that $\{ G^\vee \} \rb_S V$ implies $V \in F$ for all $V$.
The condition $\{ G^\vee \} \rb_S V$ implies that the collection
\[ \{ W \mid \exists W', S \cap W \subseteq W', (G^\vee \in W' \implies W' \subseteq V) \} \]
is a Cauchy cover of $X$.
Thus, there exist sets $W$ and $W'$ such that $\{ G^\vee \} \rb_X W$, $S \cap W \subseteq W'$, and $W' \subseteq V$.
By \rlem{filter-point-char}, there exists $W'' \in G$ such that $W'' \rb_X W'$.
It follows that $V \in F$.

Now, let us prove the converse.
Suppose that $S$ is complete and let $x$ be a limit point of $S$.
Define $F = \{ V \mid \exists W, \{ x \} \rb_X W, U \cap S \subseteq V \}$.
Then $F$ is a Cauchy filter on $S$.
The fact that $F$ is proper follows from the assumption that $x$ is a limit point of $S$, and the other properties of Cauchy filters are straightforward to verify.
To show that $x \in S$, it suffices to prove that $x = F^\vee$.

By \rlem{separated-char}, it is enough to show that any pair of neighborhoods of $x$ and $F^\vee$ intersect.
Let $U$ be a neighborhood of $x$, and let $V$ be a neighborhood of $F^\vee$.
By \rlem{filter-point-char}, $V \in F$, which means there exists a set $W$ such that $\{ x \} \rb_X W$ and $W \cap S \subseteq V$.
Since $x$ is a limit point of $S$ and $U \cap W$ is a neighborhood of $x$, there exists a point $y \in U \cap W \cap S$.
Thus, $U$ and $V$ intersect, completing the proof.
\end{proof}

\begin{cor}[closed-compact]
A subspace of a complete cover space is compact if and only if it is closed and totally bounded.
\end{cor}

As a consequence, we obtain a constructive version of the Heine-Borel theorem:

\begin{cor}
A subspace of $\mathbb{R}^n$ is compact if and only if it is closed and bounded.
\end{cor}
\begin{proof}
This follows from \rcor{tb-metric} and \rcor{closed-compact}.
\end{proof}

\section{Limits}
\label{sec:limits}

Limits of functions are closely related to the lifting problem for cover spaces.
In this section, we show how various types of limits can be expressed in terms of liftings.
Generally, a dense cover map $X \to \widetilde{X}$ represents a form of limit.
The data of a limit in a precover space $Y$ is given by a function $f : X \to Y$.
We say that the limit of $f$ exists if this function extends to a cover map $\widetilde{f} : \widetilde{X} \to Y$.
The limit itself is then given by the image of $\widetilde{f}$ at some point in $\widetilde{X}$.

To demonstrate, let us consider sequential limits.
First, we define a uniform space structure on $\mathbb{N}$.
A cover $C$ of $\mathbb{N}$ is uniform if there exist $N \in \mathbb{N}$ and $U \in C$ such that $U$ contains all $n \geq N$.
It is straightforward to verify that this indeed defines a uniform space structure on $\mathbb{N}$.

Cauchy filters in this space correspond to what are known as extended natural numbers.
An \emph{extended natural number} is a sequence $x$ of binary digits, i.e., elements of $\{ 0, 1 \}$, such that $x_n = 1$ for at most one $n \in \mathbb{N}$.
The set of extended natural numbers is denoted by $\mathbb{N}_\infty$.
There exists a function $\mathbb{N} \amalg \{ \infty \} \to \mathbb{N}_\infty$, which maps $n \in \mathbb{N}$ to the sequence $x$ with $x_n = 1$, and maps $\infty$ to the sequence consisting entirely of zeros.
Classically, this function is bijective, but constructively, it is only injective.

\begin{prop}
There is a bijection between the set of regular Cauchy filters on $\mathbb{N}$ and the set of extended natural numbers.
\end{prop}
\begin{proof}
A proper filter $F$ on $\mathbb{N}$ is Cauchy if and only if, for every $N \in \mathbb{N}$, either there exists $n < N$ such that $\{ n \} \in F$ or $[N,\infty) \in F$.
Let us show that the membership $\{ n \} \in F$ is decidable.
We proceed by induction on $n$.

If $\{ k \} \in F$ for some $k < n$, then $\{ n \} \notin F$ since $F$ is proper.
Otherwise, either $\{ n \} \in F$ or $[N + 1, \infty) \in F$.
The latter case implies $\{ n \} \notin F$.
Hence, we can decide whether $\{ n \} \in F$.

Now, we define a function $s$ from the set of Cauchy filters on $\mathbb{N}$ to $\mathbb{N}_\infty$ as follows:
\[ 
s(F)_n = 
\begin{cases} 
1 & \text{if } \{ n \} \in F, \\
0 & \text{otherwise}.
\end{cases}
\]
Since $\{ n \} \in F$ for at most one $n$, the sequence $s(F)$ is indeed an extended natural number.

Conversely, we define a function $f$ from $\mathbb{N}_\infty$ to the set of Cauchy filters on $\mathbb{N}$ by:
\[ f(x) = \{ U \mid (\exists n \in U, x_n = 1) \lor (\exists N \in \mathbb{N}, [N,\infty) \subseteq U, \forall n < N, x_n = 0) \}. \]
It is straightforward to verify that $f(x)$ is indeed a Cauchy filter.

Finally, we observe that $s(f(x)) = x$ and $f(s(F)) \subseteq F$, which implies $f(s(F))$ is equivalent to $F$.
Also, it is clear that $s(F) = s(G)$ whenever $F \subseteq G$, which implies that $s$ respects the equivalence of Cauchy filters.
Thus, we get a bijection between $\mathbb{N}_\infty$ and the equivalence classes of Cauchy filters.
Since regular Cauchy filters are unique representatives of their equivalence classes, we obtain the desired bijection.
\end{proof}

We say that a sequence $f : \mathbb{N} \to X$ in a precover space $X$ is \emph{Cauchy} if, for every Cauchy cover $C$ of $X$, there exist $N \in \mathbb{N}$ and $U \in C$ such that, for every $n \geq N$, $f_n \in U$.
If $X$ is a metric space, this definition is equivalent to the usual notion of a Cauchy sequence.
Thus, this approach generalizes Cauchy sequences to arbitrary precover spaces.

Moreover, $f : \mathbb{N} \to X$ is a Cauchy sequence if and only if it is a cover map with respect to the uniform structure on $\mathbb{N}$ that we defined earlier.
If $X$ is a complete cover space, then \rthm{dense-lift} implies that $f : \mathbb{N} \to X$ is a Cauchy sequence if and only if it extends to a cover map $\widetilde{f} : \mathbb{N}_\infty \to X$.
In this case, the limit of $f$ is defined as $\widetilde{f}(\infty)$.

This concept can be generalized to arbitrary directed sets, leading to a more general notion of limits \cite{limits}.
A \emph{directed set} $I$ is a preordered set in which every finite subset has an upper bound.
For an element $N \in I$, we write $I_{\geq N}$ to denote the set $\{ n \in I \mid n \geq N \}$.
There is a precover space structure on $I$ where a cover $C$ is Cauchy if and only if there exist $U \in C$ and $N \in I$ such that $I_{\geq N} \subseteq U$.
It is straightforward to verify that this definition satisfies \axref{CT}, \axref{CE}, and \axref{CG}, though it may not satisfy \axref{CR}.
In general, the cover space structure on $I$ is defined as the reflection of this precover space structure.

\begin{defn}
Let $I$ be a directed set and $X$ a cover space.
We say that a function $f : I \to X$ \emph{converges} if it is a cover map.
\end{defn}

A \emph{limit} of a function $f : I \to X$ is a point $x \in X$ such that, for every neighborhood $U$ of $x$, there exists $N \in I$ such that $I_{\geq N} \subseteq f^{-1}(U)$.
If a function has a limit, then it converges.
Conversely, if $X$ is complete and $f$ converges, then it has a limit.
Indeed, the completion $I_\infty$ of $I$ contains a point $\infty$ corresponding to the Cauchy filter $\{ U \mid \exists N \in I, I_{\geq N} \subseteq U \}$.
If $X$ is complete and $f$ converges, then $f$ extends to a function from the completion of $I$ to $X$, and the limit of $f$ is given by the image of $\infty$.

The following proposition provides an explicit characterization of convergence:

\begin{prop}[conv-char]
Let $I$ be a directed set and $X$ a cover space.
Then a function $f : I \to X$ converges if and only if, for every Cauchy cover $C$ of $X$, there exist $U \in C$ and $N \in I$ such that $I_{\geq N} \subseteq f^{-1}(U)$.
\end{prop}
\begin{proof}
Since $X$ is a cover space, $f$ converges if and only if, for every Cauchy cover $C$ of $X$, the family $\{ f^{-1}(U) \mid U \in C \}$ is a Cauchy cover in the precover space structure on $I$.
Unpacking the definitions, we obtain the required condition.
\end{proof}

If $X$ is a uniform space, we can simplify the condition in the previous proposition:

\begin{prop}
Let $I$ be a directed set and $X$ a uniform space.
Then a function $f : I \to X$ converges if and only if, for every uniform cover $C$ of $X$, there exist $U \in C$ and $N \in I$ such that $I_{\geq N} \subseteq f^{-1}(U)$.
\end{prop}
\begin{proof}
If $f$ converges, then the required condition holds by the previous proposition.
Conversely, assume that this condition holds, and let us prove that $f$ converges.
Define the filter $F = \{ U \mid \exists N \in I, \forall n \geq N, f(n) \in U \}$.
By assumption, this proper filter intersects with every uniform cover of $X$.
By \rprop{cauchy-filter}, $F$ also intersects with every Cauchy cover of $X$.
Hence, $f$ is a cover map, as required.
\end{proof}

This result gives the usual definition of convergence for metric spaces:

\begin{cor}
Let $I$ be a directed set and $X$ a metric space.
Then a function $f : I \to X$ converges if and only if, for every $\varepsilon > 0$, there exists $N \in I$ such that $d(f(n),f(N)) < \varepsilon$ for all $n \geq N$.
\end{cor}

Now, we discuss limits of sequences of functions.
It is well-known that the pointwise limit of a sequence of continuous maps may not be continuous.
The same issue arises for cover maps.
This limitation becomes clear when expressed in terms of the lifting property.

Indeed, a sequence of functions $X \to Y$ can be represented by a single function $f : \mathbb{N} \times X \to Y$.
Even if $f(n,-)$ is a cover map for all $n \in \mathbb{N}$ and $f(-,x)$ is a cover map for all $x \in X$, this does not imply that $f$ is a cover map.
In other words, if $f$ represents a pointwise convergent sequence of cover maps, it does not guarantee that $f$ itself is a cover map.

Thus, it is natural to require that the corresponding function $\mathbb{N} \times X \to Y$ is a cover map.
More generally, for an arbitrary directed set $I$, we might consider cover maps of the form $I \times X \to Y$. 
As before, if $Y$ is a complete cover space, then $f : I \times X \to Y$ is a cover map if and only if it extends to a map $\widetilde{f} : I_\infty \times X \to Y$.
In this case, the limit of $f$ is defined as $\widetilde{f}(\infty, -)$.

It will be convenient to have an explicit description of Cauchy covers of $I \times X$ for the following discussion.
The following lemma provides such a description:

\begin{lem}[dir-prod-char]
Let $I$ be a directed set equipped with the precover space structure, and let $X$ be a cover space.
Then a set $C$ is a Cauchy cover of $I \times X$ if and only if the following conditions hold:
\begin{enumerate}
\item \label{it:dir-char} The set $\{ U' \mid \exists N \in I, \exists U \in C, I_{\geq N} \times U' \subseteq U \}$ is a Cauchy cover of $X$.
\item \label{it:dir-cover} For every $n \in I$, the set $\{ U' \mid \exists U \in C, \forall x \in U', (n,x) \in U \}$ is a Cauchy cover of $X$.
\end{enumerate}
\end{lem}
\begin{proof}

We begin by proving that covers satisfying these two conditions form a precover space structure on $I \times X$.
Axioms \axref{CT} and \axref{CE} are easily verified, so we focus on checking axiom \axref{CG}.
Let $C$ be a set, and let $\{ D_U \}_{U \in C}$ be a collection of sets such that $C$ and $D_U$ satisfy conditions \eqref{it:dir-char} and \eqref{it:dir-cover}.
We need to show that $E = \{ U \cap V \mid U \in C, V \in D_U \}$ also satisfies these conditions.

\textbf{Condition \eqref{it:dir-char}:}
Define
\begin{align*}
C' & = \{ U' \mid \exists N \in I, \exists U \in C, I_{\geq N} \times U' \subseteq U \}, \\
D'_{U'} & = \{ V' \mid \exists N \in I, \exists U \in C, I_{\geq N} \times U' \subseteq U, \exists V \in D_U, I_{\geq N} \times V' \subseteq V \}.
\end{align*}
By assumption, $C'$ is a Cauchy cover of $X$.
Clearly, $D'_{U'}$ is also a Cauchy cover of $X$ for every $U' \in C'$.
Since $\{ U' \cap V' \mid U' \in C', V' \in D'_{U'} \}$ is a subset of $\{ W' \mid \exists N \in I, \exists W \in E, I_{\geq N} \times W' \subseteq W \}$, it follows that the latter set is a Cauchy cover of $X$.
Thus, \eqref{it:dir-char} holds for $E$.

\textbf{Condition \eqref{it:dir-cover}:}
Let $n \in I$. Define
\begin{align*}
C' & = \{ U' \mid \exists U \in C, \forall x \in U', (n,x) \in U \}, \\
D'_{U'} & = \{ V' \mid \exists U \in C, \exists V \in D_U, (\forall x \in U', (n,x) \in U), (\forall x \in V', (n,x) \in V) \}.
\end{align*}
Clearly, $C'$ and $D'_{U'}$ are Cauchy covers of $X$.
Since $\{ U' \cap V' \mid U' \in C', V' \in D'_{U'} \}$ is a subset of $\{ W' \mid \exists W \in E, \forall x \in W', (n,x) \in W \}$, the latter set is a Cauchy cover of $X$.
Thus, \eqref{it:dir-cover} holds for $E$.

To prove that every Cauchy cover of $I \times X$ satisfies conditions \eqref{it:dir-char} and \eqref{it:dir-cover}, it suffices to consider basic Cauchy covers,
i.e., covers of the form $\{ \pi_1^{-1}(U) \mid U \in C \}$ and $\{ \pi_2^{-1}(V) \mid V \in D \}$, where $C$ is a Cauchy cover of $I$ and $D$ is a Cauchy cover of $X$.

Covers of the form $\{ \pi_1^{-1}(U) \mid U \in C \}$ satisfy the conditions because both of the following covers are refined by $\{ \top \}$:
\begin{align*}
& \{ U' \mid \exists N \in I, \exists U \in C, I_{\geq N} \times U' \subseteq \pi_1^{-1}(U) \}, \\
& \{ U' \mid \exists U \in C, \forall x \in U', (n,x) \in \pi_1^{-1}(U) \}.
\end{align*}
Similarly, covers of the form $\{ \pi_2^{-1}(V) \mid V \in D \}$ satisfy the conditions because both of the following covers are refined by $D$:
\begin{align*}
& \{ V' \mid \exists N \in I, \exists V \in D, I_{\geq N} \times V' \subseteq \pi_2^{-1}(V) \}, \\
& \{ V' \mid \exists V \in D, \forall x \in V', (n,x) \in \pi_2^{-1}(V) \}.
\end{align*}

Finally, we prove that every set $C$ satisfying conditions \eqref{it:dir-char} and \eqref{it:dir-cover} is a Cauchy cover of $I \times X$.
Define
\begin{align*}
C' & = \{ U' \mid \exists N \in I, \exists U \in C, I_{\geq N} \times U' \subseteq U \}, \\
D & = \{ \pi_2^{-1}(U') \mid U' \in C' \}, \\
E_V & = \{ W \mid \exists U' \subseteq X, \exists N \in I, \exists U \in C, I_{\geq N} \times U' \subseteq U, V = \pi_2^{-1}(U'), \\
& \qquad \qquad (W \subseteq \pi_1^{-1}(I_{\geq N})) \lor (\exists W_1 \in C, W \subseteq W_1) \}.
\end{align*}

Since $\{ V \cap W \mid V \in D, W \in E_V \}$ refines $C$, and $D$ is a Cauchy cover of $I \times X$, it suffices to check that $E_V$ is a Cauchy cover of $I \times X$ for every $V \in D$.
Let $U' \in C'$, $N \in I$, and $U \in C$ such that $V = \pi_2^{-1}(U')$ and $I_{\geq N} \times U' \subseteq U$.
Define
\begin{align*}
F' & = \{ I_{\geq N} \} \cup \{ \{ n \} \mid n \in I \}, \\
F & = \{ \pi_1^{-1}(S') \mid S' \in F' \}, \\
G_S & = \{ T \mid \exists S' \subseteq I, S = \pi_1^{-1}(S'), (S' = I_{\geq N}) \lor \\
& \qquad \qquad (\exists n \in I, S' = \{ n \}, \exists U_1' \subseteq X, \exists U_1 \in C, \{ n \} \times U_1' \subseteq U_1, T = \pi_2^{-1}(U_1')) \}.
\end{align*}

Since $\{ S \cap T \mid S \in F, T \in G_S \} \subseteq E_V$ and $F$ is a Cauchy cover of $I \times X$,
it suffices to verify that $G_S$ is a Cauchy cover of $I \times X$ for every $S \in F$.
Let $S' \in F'$ such that $S = \pi_1^{-1}(S')$.
Then either $S' = I_{\geq N}$ or $S' = \{ n \}$ for some $n \in I$.
In the former case, $G_S$ is Cauchy since it is refined by $\{ \top \}$.
In the latter case, $G_S$ is Cauchy since it is refined by the Cauchy cover $\{ \pi_2^{-1}(U'') \mid U'' \in C'' \}$,
where $C''$ is the Cauchy cover $\{ U'' \mid \exists U \in C, \forall x \in U'', (n,x) \in U \}$.
This completes the proof.
\end{proof}

The following proposition provides a characterization for the convergence of functions of the form $I \times X \to Y$:

\begin{prop}[func-conv-char]
Let $I$ be a directed set, and let $X$ and $Y$ be cover spaces.
A function $f : I \times X \to Y$ is a cover map if and only if the following conditions hold:
\begin{enumerate}
\item For every $n \in I$, the map $f(n, -) : X \to Y$ is a cover map.
\item For every Cauchy cover $D$ of $Y$, the set
        \[ \{ U \mid \exists N \in I, \exists V \in D, I_{\geq N} \times U \subseteq f^{-1}(V) \} \]
        is a Cauchy cover of $X$.
\end{enumerate}
\end{prop}
\begin{proof}
Since the reflector preserves products, the function $f : I \times X \to Y$ is a cover map for the cover space structure on $I$ if and only if it is a cover map for the precover space structure on $I$.
Thus, $f$ is a cover map if and only if, for every Cauchy cover $D$ on $Y$, the collection $\{ f^{-1}(V) \mid V \in D \}$ satisfies conditions \eqref{it:dir-char} and \eqref{it:dir-cover} of \rlem{dir-prod-char}.
These conditions are equivalent to the ones stated in the proposition, completing the proof.
\end{proof}

If $Y$ is a uniform space, we can improve the characterization as follows:

\begin{prop}
Let $I$ be a directed set, $X$ a cover space, and $Y$ a uniform space.
Then $f : I \times X \to Y$ is a cover map if and only if the following conditions hold:
\begin{enumerate}
\item For every $n \in I$, the map $f(n, -) : X \to Y$ is a cover map.
\item For every uniform cover $D$ of $Y$, the set
        \[ \{ U \mid \exists N \in I, \forall x \in U, \exists V \in D, \forall n \geq N, f(n, x) \in V \} \]
        is a Cauchy cover of $X$.
\end{enumerate}
\end{prop}
\begin{proof}
The ``only if'' direction follows immediately from \rprop{func-conv-char}.
To prove the converse, we show that, for every uniform cover $E$ of $Y$, there exists a Cauchy cover $C$ of $X$ such that, for every $U \in C$,
there exist $N \in I$ and a Cauchy cover $D$ of $X$ with the property that, for every $V \in D$, there exists $W \in E$ satisfying
$I_{\geq N} \times (U \cap V) \subseteq f^{-1}(W)$.

By \axref{UU}, there exists a uniform cover $E'$ of $Y$ such that, for every $W' \in E'$, there exists $W \in E$ such that, for all $W'' \in E'$, we know $\overlap{W''}{W'} \implies W'' \subseteq W$.
By assumption, the set
\[ C = \{ U \mid \exists N \in I, \forall x \in U, \exists W' \in E', \forall n \geq N, f(n, x) \in W' \} \]
is a Cauchy cover of $X$.
Let $U \in C$.
Then there exists $N \in I$ such that
\[ \forall x \in U, \exists W' \in E', \forall n \geq N, f(n, x) \in W'. \]
Define the set
\[ D = \{ f(N, -)^{-1}(W') \mid W' \in E' \}, \]
which is a Cauchy cover by assumption.

We now show that, for every $W' \in E'$, there exists $W \in E$ such that
\[ I_{\geq N} \times (U \cap f(N, -)^{-1}(W')) \subseteq f^{-1}(W). \]
Let $W \in E$ satisfy the condition that, for every $W'' \in E'$, we have
$\overlap{W''}{W'} \implies W'' \subseteq W$.
Let $n \in I$ with $n \geq N$ and $x \in U$ such that $f(N, x) \in W'$. Then there exists $W'' \in E'$ such that, for every $n \geq N$, $f(n, x) \in W''$.
Since $f(N, x) \in W'' \cap W'$, it follows that $W'' \subseteq W$. Thus, $f(n, x) \in W$, completing the proof of this intermediate step.

Finally, we prove that $f : I \times X \to Y$ is a cover map.
It suffices to show that, for every uniform cover $E$ of $Y$, the set $\{ f^{-1}(W) \mid W \in E \}$ is a Cauchy cover of $I \times X$.
By Lemma~\ref{lem:dir-prod-char}, we need to show that the set
\[
C' = \{ U \mid \exists N \in I, \exists W \in E, I_{\geq N} \times U \subseteq f^{-1}(W) \}
\]
is a Cauchy cover of $X$.

Let $C$ be a Cauchy cover of $X$ such that, for every $U \in C$, there exist $N \in I$ and a Cauchy cover $D$ of $X$ satisfying
\[
\forall V \in D, \exists W \in E, I_{\geq N} \times (U \cap V) \subseteq f^{-1}(W).
\]
Then, for every $U \in C$, the set
\[
D_U = \{ V \mid \exists N \in I, \exists W \in E, I_{\geq N} \times (U \cap V) \subseteq f^{-1}(W) \}
\]
is a Cauchy cover. This completes the proof since $C'$ is refined by the Cauchy cover
$\{ U \cap V \mid U \in C, V \in D_U \}$.
\end{proof}

\begin{cor}[func-conv-metric-char]
Let $I$ be a directed set, $X$ a cover space, and $Y$ a metric space.
Then $f : I \times X \to Y$ is a cover map if and only if the following conditions hold:
\begin{enumerate}
\item For every $n \in I$, the map $f(n, -) : X \to Y$ is a cover map.
\item For every $\varepsilon > 0$, the set
        \[ \{ U \mid \exists N \in I, \forall x \in U, \forall n \geq N, d_Y(f(n, x), f(N, x)) < \varepsilon \} \]
        is a Cauchy cover of $X$.
\end{enumerate}
\end{cor}

If $(a_n)$ is a sequence of real numbers, we say that the series $\sum\limits_{n=0}^\infty a_n$ \emph{converges} if the sequence of partial sums, $s_k = \sum\limits_{n=0}^k a_n$, converges.
If the power series $\sum\limits_{n=0}^\infty a_n x^n$ converges for all $x$, then standard techniques imply that it converges uniformly on bounded subsets.
By \rcor{func-conv-metric-char}, it follows that the function $x \mapsto \sum\limits_{n=0}^\infty a_n x^n$ is a cover map in this case.
This result can be used to define various cover maps such as the exponential function, sine, cosine, and other similar functions.

Constructing these maps as morphisms of locales, however, may be somewhat challenging (see, for example, \cite{vickers-exp}).
Since cover maps correspond to localic maps, the framework of cover spaces provides an alternative approach for the construction of such maps.
This approach is much closer to the usual classical presentation.

Finally, we discuss the limits of functions of the form $f : \mathbb{R}_* \to \mathbb{R}$.
As usual, we say that $y$ is a \emph{limit} of $f$ at $0$ if, for every $\varepsilon > 0$,
there exists $\delta > 0$ such that, for every $x \in \mathbb{R}_*$, if $\abs{x} < \delta$, then $\abs{f(x) - y} < \varepsilon$.
The relationship between this notion and a lifting property is given by the following proposition:

\begin{prop}
Let $f : \mathbb{R}_* \to \mathbb{R}$ be a cover map.
Then $f$ has a limit at $0$ if and only if $f$ extends to a cover map $\widetilde{f} : \mathbb{R} \to \mathbb{R}$.
The limit is then equal to $\widetilde{f}(0)$.
\end{prop}
\begin{proof}
Assume that $f$ extends to $\widetilde{f}$.
By \rlem{cover-map-rb}, for every $\varepsilon > 0$, the set $\widetilde{f}^{-1}(B_\varepsilon(\widetilde{f}(0)))$ is a neighborhood of $0$.
It follows that $\widetilde{f}(0)$ is the limit of $f$ at $0$.

Now assume that $L$ is the limit of $f$ at $0$.
Let $\mathbb{R}'_*$ denote the set of invertible real numbers, equipped with the transferred cover space structure.
Since $\mathbb{R}'_*$ embeds densely into $\mathbb{R}$, every cover map $\mathbb{R}'_* \to \mathbb{R}$ uniquely extends to a map $\mathbb{R} \to \mathbb{R}$.
Thus, it suffices to show that $f : \mathbb{R}_* \to \mathbb{R}$ is a cover map when regarded as a map $\mathbb{R}'_* \to \mathbb{R}$.
That is, we need to prove that, for every $\varepsilon > 0$, the family $\{ f^{-1}(B_\varepsilon(x)) \mid x \in \mathbb{Q} \}$ is a Cauchy cover of $\mathbb{R}'_*$.

For every $\varepsilon > 0$, there exists $\delta > 0$ such that $B_\delta(0) \cap \mathbb{R}_* \subseteq f^{-1}(B_\varepsilon(L))$.
Now, define $V = \left(-\infty, -\frac{\delta}{3}\right) \cup \left(\frac{\delta}{3}, \infty\right)$, equipped with the transferred cover space structure.
It does not matter whether this structure is transferred from $\mathbb{R}_*$ or $\mathbb{R}$, as the result is the same.
Since the composite $V \to \mathbb{R}_* \overset{f}{\to} \mathbb{R}$ is a cover map,
there exists a Cauchy cover $C$ of $\mathbb{R}$ such that $\{ U \cap V \mid U \in C \}$ refines $\{ f^{-1}(B_\varepsilon(y)) \mid y \in \mathbb{Q} \}$.

Since $\{ U \cap B_{\delta/3}(x) \cap \mathbb{R}_* \mid U \in C, x \in \mathbb{Q} \}$ is a Cauchy cover of $\mathbb{R}'_*$,
it is enough to show that it refines $\{ f^{-1}(B_\varepsilon(x)) \mid x \in \mathbb{Q} \}$.
Let $U \in C$ and $x \in \mathbb{Q}$.
Consider two cases:
\begin{itemize}
\item If $\abs{x} \geq \frac{2}{3} \delta$, then $U \cap B_{\delta/3}(x) \cap \mathbb{R}_* \subseteq U \cap V$,
so $U \cap B_{\delta/3}(x) \cap \mathbb{R}_* \subseteq f^{-1}(B_\varepsilon(y))$ for some $y \in \mathbb{Q}$.
\item If $\abs{x} \leq \frac{2}{3} \delta$, then $U \cap B_{\delta/3}(x) \cap \mathbb{R}_* \subseteq B_\delta(0) \cap \mathbb{R}_* \subseteq f^{-1}(B_\varepsilon(L))$.
\end{itemize}
This completes the proof.
\end{proof}

\bibliographystyle{amsplain}
\bibliography{ref}

\end{document}